\documentclass[12pt]{amsart}

\usepackage{mathptmx} 
\DeclareSymbolFont{cmletters}{OML}{cmm}{m}{it}              
\DeclareSymbolFont{cmsymbols}{OMS}{cmsy}{m}{n}
\DeclareSymbolFont{cmlargesymbols}{OMX}{cmex}{m}{n}
\DeclareMathSymbol{\myjmath}{\mathord}{cmletters}{"7C}     \let\jmath\myjmath 
\DeclareMathSymbol{\myamalg}{\mathbin}{cmsymbols}{"71}     \let\amalg\myamalg
\DeclareMathSymbol{\mycoprod}{\mathop}{cmlargesymbols}{"60}\let\coprod\mycoprod
\DeclareMathSymbol{\myalpha}{\mathord}{cmletters}{"0B}     \let\alpha\myalpha 
\DeclareMathSymbol{\mybeta}{\mathord}{cmletters}{"0C}      \let\beta\mybeta
\DeclareMathSymbol{\mygamma}{\mathord}{cmletters}{"0D}     \let\gamma\mygamma
\DeclareMathSymbol{\mydelta}{\mathord}{cmletters}{"0E}     \let\delta\mydelta
\DeclareMathSymbol{\myepsilon}{\mathord}{cmletters}{"0F}   \let\epsilon\myepsilon
\DeclareMathSymbol{\myzeta}{\mathord}{cmletters}{"10}      \let\zeta\myzeta
\DeclareMathSymbol{\myeta}{\mathord}{cmletters}{"11}       \let\eta\myeta
\DeclareMathSymbol{\mytheta}{\mathord}{cmletters}{"12}     \let\theta\mytheta
\DeclareMathSymbol{\myiota}{\mathord}{cmletters}{"13}      \let\iota\myiota
\DeclareMathSymbol{\mykappa}{\mathord}{cmletters}{"14}     \let\kappa\mykappa
\DeclareMathSymbol{\mylambda}{\mathord}{cmletters}{"15}    \let\lambda\mylambda
\DeclareMathSymbol{\mymu}{\mathord}{cmletters}{"16}        \let\mu\mymu
\DeclareMathSymbol{\mynu}{\mathord}{cmletters}{"17}        \let\nu\mynu
\DeclareMathSymbol{\myxi}{\mathord}{cmletters}{"18}        \let\xi\myxi
\DeclareMathSymbol{\mypi}{\mathord}{cmletters}{"19}        \let\pi\mypi
\DeclareMathSymbol{\myrho}{\mathord}{cmletters}{"1A}       \let\rho\myrho
\DeclareMathSymbol{\mysigma}{\mathord}{cmletters}{"1B}     \let\sigma\mysigma
\DeclareMathSymbol{\mytau}{\mathord}{cmletters}{"1C}       \let\tau\mytau
\DeclareMathSymbol{\myupsilon}{\mathord}{cmletters}{"1D}   \let\upsilon\myupsilon
\DeclareMathSymbol{\myphi}{\mathord}{cmletters}{"1E}       \let\phi\myphi
\DeclareMathSymbol{\mychi}{\mathord}{cmletters}{"1F}       \let\chi\mychi
\DeclareMathSymbol{\mypsi}{\mathord}{cmletters}{"20}       \let\psi\mypsi
\DeclareMathSymbol{\myomega}{\mathord}{cmletters}{"21}     \let\omega\myomega
\DeclareMathSymbol{\myvarepsilon}{\mathord}{cmletters}{"22}\let\varepsilon\myvarepsilon
\DeclareMathSymbol{\myvartheta}{\mathord}{cmletters}{"23}  \let\vartheta\myvartheta
\DeclareMathSymbol{\myvarpi}{\mathord}{cmletters}{"24}     \let\varpi\myvarpi
\DeclareMathSymbol{\myvarrho}{\mathord}{cmletters}{"25}    \let\varrho\myvarrho
\DeclareMathSymbol{\myvarsigma}{\mathord}{cmletters}{"26}  \let\varsigma\myvarsigma
\DeclareMathSymbol{\myvarphi}{\mathord}{cmletters}{"27}    \let\varphi\myvarphi

\usepackage{amsthm,amsmath,amssymb,amscd,graphics,enumerate,latexsym,verbatim,stmaryrd,graphicx,color}
\usepackage[all]{xy}
\usepackage{tikz}
\usetikzlibrary{matrix,arrows}

\theoremstyle{plain}
\newtheorem{thm}{Theorem}[section]
\newtheorem{cor}[thm]{Corollary}
\newtheorem{lemma}[thm]{Lemma}
\newtheorem{prop}[thm]{Proposition}
\newtheorem{conj}[thm]{Conjecture}

\theoremstyle{definition}

\newtheorem{rem}[thm]{Remark}
\newtheorem{ex}[thm]{Example}

\DeclareMathOperator{\GL}{GL}

\DeclareMathOperator{\Gr}{Gr}

\DeclareMathOperator{\Spec}{Spec}

\DeclareMathOperator{\Hom}{Hom}
\DeclareMathOperator{\Ext}{Ext}

\DeclareMathOperator{\Adm}{Adm}

\DeclareMathOperator{\vspan}{span}

\def\decomp{\coprod^\circ}

\def\0{{\bf 0}}
\def\A{{\mathbb A}}

\def\C{{\mathbb C}}
\def\F{{\mathbb F}}
\def\G{{\mathbb G}}

\def\N{{\mathbb N}}
\def\P{{\mathbb P}}

\def\Z{{\mathbb Z}}

\def\cB{{\mathcal B}}

\def\cM{{\mathcal M}}

\def\Fun{{\F_1}}

\def\int{\textup{int}}

\def\pr{\textup{pr}}
\def\id{\textup{id}}
\def\1{\textbf{1}}

\def\ue{{\underline{e}}}

\def\um{{\underline{m}}}

\def\udim{{\underline{\dim}\, }}

\def\rk{{\textup{rk}}}

\def\={\equiv}
\def\n={\equiv\hspace{-10,5pt}/\hspace{3,5pt}}

\def\red{{\textup{red}}}

\newdir{ >}{{}*!/-5pt/@^{(}}\newcommand{\arincl}[1]{\ar@{ >->}@<-0,0ex>#1} 

\newcommand{\norm}[1]{\left| #1 \right|}

\newcommand{\tinyvec}[2]{\bigl( \begin{smallmatrix} #1 \\ #2  \end{smallmatrix} \bigr)}
\newcommand{\tinymat}[4]{\bigl( \begin{smallmatrix} #1 & #2 \\ #3 & #4 \end{smallmatrix} \bigr)}
\newcommand{\tinybmat}[4]{\bigl[ \begin{smallmatrix} #1 & #2 \\ #3 & #4 \end{smallmatrix} \bigr]}

\begin{document}

\title[On Schubert decompositions of quiver Grassmannians]{On Schubert decompositions of quiver Grassmannians}
\author{Oliver Lorscheid}
\date{}
\address{IMPA, Estrada Dona Castorina 110, 22460-320 Rio de Janeiro, Brazil}
\email{lorschei@impa.br}

\begin{abstract}
 In this paper, we introduce Schubert decompositions for quiver Grassmannians and investigate example classes of quiver Grassmannians with a Schubert decomposition into affine spaces. The main theorem puts the cells of a Schubert decomposition into relation to the cells of a certain simpler quiver Grassmannian. This allows us to extend known examples of Schubert decompositions into affine spaces to a larger class of quiver Grassmannians. This includes exceptional representations of the Kronecker quiver as well as representations of forests with block matrices of the form $\tinymat 0100$. Finally, we draw conclusions on the Euler characteristics and the cohomology of quiver Grassmannians.
\end{abstract}

\maketitle

\begin{footnotesize}\tableofcontents\end{footnotesize}


\section*{Introduction}
\label{intro}

In 1994, Lusztig published his seminal book \cite{Lusztig94} on the existence of canonical bases for Lie algebras. This was the starting point of vivid research that aimed for a better understanding of canonical bases. Despite being hard to compute, much insight was gained into the general structure of canonical bases during the last years. 

A major contribution to the subject was the introduction of cluster algebras by Fomin and Zelevinsky in 2002, see \cite{fomin-Zelevinsky02} and their subsequent publications. An important feature of the theory of cluster algebras is the mutation operation that associates to a quiver $Q$, by recursion, a set of so-called cluster variables, which generates the associated cluster algebra. In 2006, Caldero and Chapoton found an explicit formula that expresses the cluster variables in terms of the Euler characteristics of the quiver Grassmannians $\Gr_\ue(M)$ for the rigid representations $M$ of the quiver $Q$, see \cite{Caldero-Chapoton06}.

The Caldero-Chapoton formula drew attention to quiver Grassmannians and, in particular, to their Euler characteristics. In \cite{Caldero-Reineke08}, Caldero and Reineke established many basic properties of quiver Grassmannians for acyclic quivers, e.g.\ its smoothness in the case of a rigid representation $M$. They determine the Euler characteristic of $\Gr_\ue(M)$ if $M$ is an indecomposable representation of the Kronecker quiver, and they remark that Schubert decompositions of quiver Grassmannians might help to compute their Euler characteristics.

Many other publications followed. Cerulli and Esposito inspect in \cite{Cerulli-Esposito11} quiver Grassmannians of Kronecker type in further detail and apply this to the canonical basis of cluster algebras of types $A_1^{(1)}$ and $A_2^{(1)}$. In particular, they describe a decomposition of $\Gr_\ue(M)$ into affine spaces in case $M$ is a regular representation (cf.\ Example \ref{ex: kronecker quiver}). Sz\'ant\'o establishes in \cite{Szanto11} a counting polynomial of the $\F_q$-rational points of quiver Grassmannians of Kronecker type, which hints that there exists a Schubert decomposition into affine spaces for other types of representations of the Kronecker quiver as well.

Rupel conjectures in \cite{Rupel11a} the positivity of acyclic seeds for cluster algebras. This conjecture implies that quiver Grassmannians of rigid representations have a counting polynomial in the acyclic case, which in turn implies the positivity of their Euler characteristics if the quiver Grassmannian is not empty. This conjecture was partially proven by Qin in \cite{Qin10}, followed by a complete proof by Rupel in \cite{Rupel11b}.

In \cite{Cerulli-Feigin-Reineke11}, \cite{Cerulli-Feigin-Reineke12a} and \cite{Cerulli-Feigin-Reineke12b}, Cerulli, Feigin and Reineke realize degenerate flag varieties as quiver Grassmannians of Dynkin type. A particular result of interest for the present paper is the existence of a Schubert decomposition into affine spaces (cf.\ Examples \ref{ex: degenerate_flag_varieties1} and \ref{ex: degenerate_flag_varieties2}). 

The two papers that essentially inspired the results of this paper are \cite{Cerulli11} and \cite{Haupt12}. In \cite{Cerulli11}, Cerulli gives a formula for the Euler characteristics of quiver Grassmannians of orientable string modules. In \cite{Haupt12}, Haupt extends the results of \cite{Cerulli11} to the class of tree modules and also provides a formula for the Euler characteristic of quiver Grassmannians of band modules. The method of both papers is to construct a weighted diagonal action of the one-dimensional torus $T=\G_m$ on the representation $M$ in question. This divides the quiver Grassmannian $X=\Gr_\ue(M)$ into the locally closed subscheme $X^T$ of fixed points and its complement $Z=X-X^T$, which yields
\[
 \chi\bigl(\Gr_\ue(M)\bigr) \quad = \quad \chi\bigl(X^T\bigr) \ + \ \chi\bigl(Z\bigr) \quad = \quad \chi\bigl(X^T\bigr) \ + \ \underbrace{\chi\bigl(T\bigr)}_{=0} \ \cdot \ \chi\bigl(Z/\!\!\!/T\bigr) \quad = \quad \chi\bigl(X^T\bigr).
\]
After applying this trick several times, the fixed point set $X^T$ is finite and can be identified with the number of subrepresentations of a certain quiver representation $\tilde M$ that is simpler than $M$. This means, in particular, that these Euler characteristics are positive if the quiver Grassmannian is not empty. 

During the attempt to understand the geometry of the quiver Grassmannians considered in \cite{Cerulli11} and \cite{Haupt12}, it turned out that in many cases, quiver Grassmannians have a decomposition into affine spaces. A systematic study of these decompositions led to the results of this paper. Though the methods of this paper are completely different, we will obtain formulas for the Euler characteristics of quiver Grassmannians in a class that has a large intersection with the class of cases treated in \cite{Cerulli11} and \cite{Haupt12}. The existence of Schubert decompositions into affine spaces allows us further to extract information about the cohomology. For instance, if the representation $M$ is rigid, then the Schubert cells determine an additive basis for the cohomology ring.

\subsection*{Results}
 The quiver Grassmannian $\Gr_\ue(M)$ of subrepresentations $V$ of $M$ with dimension vector $\ue$ is defined as a closed subscheme of the usual Grassmannian $\Gr(e,m)$ where $e$ is the dimension of $V$ and $m$ is the dimension of $M$ over the ground field. The intersection of $\Gr_\ue(M)$ with a Schubert decomposition of $\Gr(e,m)$ defines a Schubert decomposition of $\Gr_\ue(M)$. In general, this is not a decomposition into affine spaces, and the isomorphism type of the Schubert cells is not independent of the choices that define the Schubert decomposition for $\Gr(e,m)$. 

 The results of this paper concentrate on establishing cases of quiver Grassmannians that have a Schubert decomposition into affine spaces. The main result Theorem \ref{thm: push-forward} roughly says the following: Let $S\subset T$ be an inclusion of quivers such that the quotient $T/S$ is a tree and let $M$ be a representation of $T$. Let $F:T\to Q$ be a morphism of quivers that satisfies a certain Hypothesis (H). Then the Schubert cell $C_\beta^{F_\ast M}$ of the push-forward $F_\ast M$ of $M$ equals the product $\A^n\times C_{\beta_S}^{F_\ast M_S}$ of an affine space with the corresponding Schubert cell for the push-forward of the restriction $M_S$ of $M$ to $S$.

 While Hypothesis (H) is too technical to explain in brevity, it should be mentioned that this hypothesis is a purely combinatorial condition on the structure of the fibres of $F:T\to Q$, which can be checked easily in examples, and which can be implemented in a computer algorithm. We will illustrate a number of its consequences and other results of this paper.

\begin{enumerate}
 \item\label{part1} Let $M$ be an exceptional indecomposable representation of the Kronecker quiver and $\ue$ a dimension vector. Then $\Gr_\ue(M)$ has a Schubert decomposition into affine spaces (see Example \ref{ex: exceptional modules of the Kronecker quiver}).
 \item\label{part2} Let $T$ be a tree and $M$ a representation of $T$ whose linear maps are block matrices of the form $\tinymat 0100$ where $1$ is a square identity matrix. Then $\Gr_\ue(M)$ has a Schubert decomposition into affine spaces (see Thm.\ \ref{thm: monomial representations of forests}). If all linear maps defining $\Gr_\ue(M)$ are isomorphisms, then the quiver Grassmannian \emph{decomposes} into a series of fibre bundles whose fibres are usual Grassmannians (see Thm.\ \ref{thm: Grassmannian fibrations}).
 \item\label{part3} We re-obtain the Schubert decompositions of Cerulli and Esposito in \cite{Cerulli-Esposito11} (see Example \ref{ex: kronecker quiver}) and Cerulli, Feigin and Reineke in \cite{Cerulli-Feigin-Reineke11} (see Ex.\ \ref{ex: degenerate_flag_varieties1} and \ref{ex: degenerate_flag_varieties2}).
 \item\label{part4} If $\Gr_\ue(M,\C)=\coprod_{i\in I} X_i(\C)$ is a decomposition into complex affine spaces $X_i(\C)$, then the Euler characteristic of $\Gr_\ue(M)$ is $\chi\bigl(\Gr_\ue(M)\bigr)=\# I$ (see Prop.\ \ref{prop: euler characteristic from dias}). If $\Gr_\ue(M)$ is smooth, then the singular cohomology is concentrated in even degrees and generated by the closure of the classes of the Schubert cells (see Cor.\ \ref{cor: cohomology for smooth quiver grassmannians}). In particular, this reproduces the formulas in \cite{Cerulli11} and \cite{Haupt12} (under assumption of Hypothesis (H)) in terms of the combinatorics of the Schubert cells (see Remark \ref{rem: on the results of cerulli and haupt}). 
 \item\label{part5} If $\Gr_\ue(M,\C)=\coprod_{i\in I} X_i(\C)$ is a \emph{regular decomposition} (see Section \ref{subsection: regular decompositions}) into complex affine spaces, then the multiplication of $H^\ast(\Gr_\ue(M,\C))$ is determined by the cohomology rings of the irreducible components of $\Gr_\ue(M)$ (see Thm.\ \ref{lemma: cohomology for regular decompositions into affine spaces}).
\end{enumerate}

Next to these outcomes, the reader will find numerous side results, remarks and examples, which shall illustrate certain effects of the theory of Schubert decompositions of quiver Grassmannians.

\subsection*{Content overview}
The paper is structured as follows. In Section \ref{section: background}, we recall the definition of quiver Grassmannians and cite some basic facts.

In Section \ref{section: schubert cells}, we define Schubert cells for quiver Grassmannians. In \ref{subsection: schubert decompositions for acyclic quiver}, we explain the connection between the definition of this paper and the one given in \cite{Caldero-Reineke08} for acyclic quiver. In \ref{subsection: K-rational points}, we identify the $K$-rational points of a Schubert cell with certain matrices of generating vectors with prescribed pivot element. In \ref{subsection: examples of schubert cells}, we describe some examples of Schubert decompositions.

In Section \ref{section: tree extensions}, we introduce the notion of a tree extension $T$ of a quiver $S$. In \ref{subsection: results for tree extensions}, we state the main results for tree extension that connect a quiver Grassmannian $\Gr_\ue(M)$ of a representation $M$ of $T$ to the quiver Grassmannian $\Gr_{\ue_S}(M_S)$ of the restriction $M_S$ of $M$ to $S$. In particular, there is a smooth projective morphism $\Gr_\ue(M)\to\Gr_{\ue_S}(M_S)$ and the Schubert cells of $\Gr_\ue(M)$ are a product of a Schubert cell of $\Gr_{\ue_S}(M_S)$ with an affine space.

In Section \ref{section: push-forwards}, we introduce push-forwards of quiver representations along morphisms of quivers. In \ref{subsection: defining equations for schubert cells}, we describe the equations that are satisfied by the $K$-rational points of Schubert cells when we push-forward a representation. In \ref{subsection: comparison of M and F_ast M}, we introduce comparison morphisms between a Schubert cell and the corresponding Schubert cell for the push-forward. In \ref{subsection: relevant pairs}, we introduce relevant pairs and relevant triples, which index the variables and relations, respectively, of a Schubert cell. In \ref{subsection: fibre types}, we describe the shape of a relation of an relevant triple in dependence of the shape of the fibres of a morphism of quivers. In \ref{subsection: push-forward theorem}, we formulate Hypothesis (H), under which we can prove the main theorem of this paper (see Results above).

In Section \ref{section: consequences of the push-forward theorem}, we list some consequences of the main theorem. First of all, we give a general condition for a quiver Grassmannian to have a Schubert decomposition into affine spaces. In \ref{subsection: direct sums}, we explain a result for the quiver Grassmannian of a direct sum of representations. In \ref{subsection: monomial representations of forests}, we show that representations of forests with block matrices $\tinymat 0100$ yield quiver Grassmannians with a Schubert decomposition into affine spaces (see Results \eqref{part2}). 

In Section \ref{section: cohomology of quiver grassmannians}, we draw conclusions on the cohomology of a quiver Grassmannian that has a Schubert decompositions into affine spaces. If $\Gr_\ue(M)$ is smooth, then its cohomology classes are generated by the closure of the Schubert cells. This means in particular that the cohomology is concentrated in even degrees. Even without the smoothness assumption, we see that the Euler characteristics is given by the number of non-empty Schubert cells. In \ref{subsection: regular decompositions}, we introduce the notion of a regular decomposition. A regular decomposition allows us to deduce the multiplicative structure of the cohomology ring from the cohomology of the irreducible components. In \ref{subsection: examples and conjectures}, we describe certain example classes of quiver Grassmannians with regular Schubert decompositions and formulate two conjectures on the existence of regular Schubert decompositions.

\subsection*{Remark}
As pointed out to me by Giovanni Cerulli Irelli and Gr\'egoire Dupont, the formulas for the Euler characteristics in \cite{Cerulli11} and \cite{Haupt12} count subrepresentations that look like ``$\Fun$-rational points'' (cf.\ Szczesny's paper \cite{Szczesny12} on quiver representation over $\Fun$). That the number of $\Fun$-rational points equals the Euler characteristic is one of the main concepts in $\Fun$-geometry. Therefore, there is the hope that a better understanding of the geometry of quiver Grassmannians over $\Fun$ will help to compute their Euler characteristics. The connection of quiver Grassmannians and their Schubert decompositions to $\Fun$-geometry will be the topic of a subsequent paper. This is the reason why we work over an arbitrary base ring $k$ in this paper.

\subsection*{Acknowledgements} 
I would like to thank Giovanni Cerulli Irelli and Gr\'egoire Dupont for bringing quiver Grassmannians into my attention and for sharing their ideas on the connection to $\Fun$-geometry with me. I would like to thank Markus Reineke for his explanations and many discussions on quiver Grassmannians. I would like to thank Ethan Cotteril for our conversations on Schubert calculus. I would like to thank Damiano Testa for a discussion that helped to improve Section \ref{section: cohomology of quiver grassmannians}.


\section{Background}
\label{section: background}

A \emph{quiver} is a finite directed graph with possibly multiple edges and loops. We formalize a quiver as a quadruple $Q=(Q_0,Q_1,s,t)$ where $Q_0$ is a finite set of vertices, $Q_1$ is a finite set of arrows, $s:Q_1\to Q_0$ associates to each arrow its source or tail and $t:Q_1 \to Q_0$ associates to each arrow its target or head.

During the major part of this paper, we fix a ring $k$. We will only specify to the case $k=\C$ in some parts of Section \ref{section: cohomology of quiver grassmannians}. But for many applications, it is enough to keep the case $k=\C$ in mind.

The \emph{path algebra of $Q$ over $k$} is the $k$-algebra $k[Q]$ that is freely generated as a $k$-module by all oriented paths in $Q$. In particular, there is a path $\epsilon_p=[p|p]$ of length $0$ at every vertex $p$ of $Q$. The multiplication is defined by composition of paths if possible, and $0$ otherwise. The elements $\epsilon_p$ are idempotent, and $1=\sum_{p\in Q_0} \epsilon_p$ is the identity of $k[Q]$. As a $k$-algebra, $k[Q]$ is generated by the idempotents $\epsilon_p$ and the paths of length $1$, i.e.\ by the arrows $\alpha$ of $Q$.

A \emph{representation of $Q$ over $k$} or, for short, a \emph{$Q$-module} is a free $k[Q]$-module $M$ of finite rank. Equivalently, we can consider $M$ as a collection of free $k$-modules $M_p=\epsilon_pM$ for $p\in Q_0$ together with the collection of $k$-linear maps $M_\alpha:M_p\to M_q$, defined by $f_\alpha(\epsilon_p.m)=\alpha.m$ for every arrow $\alpha$ from $p$ to $q$. Then $M=\bigoplus_{p\in Q_0} M_p$, and the $k[Q]$-algebra structure is determined by the $k$-linear maps $M_\alpha$. The \emph{dimension vector $\udim M$ of $M$} is the tuple $\underline m=(m_p)_{p\in Q_0}$ where $m_p$ is the rank of $M_p$ over $k$.

In the following, we will relax the language a bit. We assume that the base ring is fixed and do not mention $k$ if the context is clear. We will further identify $M$ with both $\bigoplus M_p$ and $(\{M_p\}_{p\in Q_0},\{M_\alpha\}_{\alpha\in Q_1})$, and switch viewpoints where this is helpful.

A submodule $V$ of $M$ can be identified with a collection of sub-$k$-modules $V_p=\epsilon_pV$ of $M_p$ for every $p\in Q_0$ such that $M_\alpha(V_p)\subset V_q$ for every arrow $\alpha:p\to q$ in $Q_1$. Let $\ue=(e_p)_{p\in Q_0}$ be a dimension vector smaller or equal to $\um$, i.e.\ $e_p\leq m_p$ for all $p\in Q_0$. Define $\Gr_\ue(\um)$ as the product $\prod_{p\in Q_0} \Gr(e_p,m_p)$ and $R_\um(Q)$ as the product $\prod_{\alpha\in Q_1} \Hom(k^{m_{s(\alpha)}}, k^{m_{t(\alpha)}})$, which we consider as a scheme by identifying the homomorphism sets with affine spaces over $k$ of adequate dimensions. Then $Gr_\ue(\um)\times R_\um(Q)$ is a reduced scheme over $k$. The \emph{universal Grassmannian $\Gr^Q_\ue(\um)$} is the closed reduced subscheme of $Gr_\ue(\um)\times R_\um(Q)$ whose $K$-rational points are described as the set
\[
 \Bigl\{ \ \bigl(\,(V_p\subset k^{m_p})_{p\in Q_0}, \, (f_\alpha)_{\alpha\in Q_1}\,\bigr)\,\in\,Gr_\ue(\um)\times R_\um(Q)(K)  \ \Bigl| \ f_\alpha(V_p)\subset V_q\text{ for all }\alpha:p\to q\text{ in }Q_1 \ \Bigr\}
\]
for any field extension $K$ of $k$. For a $k$-rational point $M$ of $R_\um(Q)$---which is nothing else than a $Q$-module over $k$, together with a fixed basis---, the \emph{quiver Grassmannian $\Gr_\ue(M)$} is defined as the fibre of $\pr_2: \Gr_\ue^Q(\um)\to R_\um(Q)$ over $M$. See Sections 2.2 and 2.3 in \cite{Cerulli-Feigin-Reineke11} for more details on the definition of $\Gr_\ue(M)$.

Note that the isomorphism type of $\Gr_\ue(M)$ does not depend on the choice of basis for $M$, which allows us to define $\Gr_\ue(M)$ for any $Q$-module $M$. Note further that $\Gr_\ue(M)$ is in general not reduced. For a field extension $K$ of $k$, the set of $K$-rational points of $\Gr_\ue(M)$ coincides with the set
\[
 \bigl\{ \ V\subset M_K \ \bigr| \ M_\alpha(V_p)\subset V_q\text{ for all }\alpha:p\to q\text{ in }Q_1 \ \bigl\}
\]
where $M_K=M\otimes_k K$. 

The quiver Grassmannian $\Gr_\ue(M)$ is a closed subscheme of the product $\prod \Gr(e_v,m_v)$ of the usual Grassmannians over all vertices $v$ of $Q$. We cite two general facts about quiver Grassmannians.

\begin{thm}[Reineke, \cite{Reineke12}]\label{thm: reineke}
 Every projective $k$-scheme is isomorphic to a quiver Grassmannian.
\end{thm}

A $Q$-module $M$ is \emph{rigid} or \emph{exceptional} if it has no self-extensions, i.e.\ $\Ext^1(M,M)=0$. 

\begin{thm}[Caldero and Reineke, \cite{Caldero-Reineke08}]\label{thm: caldero-reineke}
 If $M$ is a rigid $Q$-module, then $\Gr_\ue(M)_k$ is a smooth $k$-scheme.
\end{thm}

\section{Schubert cells}
\label{section: schubert cells}

Let $M$ be a free $k$-module and $e\leq \rk M$ a non-negative integer. Then the choice of a (linearly) ordered basis $\cB$ of $M$ over $k$ defines a Schubert cell decomposition of the usual Grassmannian $\Gr_e(M)$ into affine spaces. In case of a $Q$-module $M$ with dimension vector $\um$ and $\ue\leq\um$, the quiver Grassmannian $\Gr_\ue(M)$ is a subscheme of the usual Grassmannian $\Gr_{\norm\ue}(M)$ via the closed embedding
\[
 \iota: \quad \Gr_\ue(M) \quad \longrightarrow \quad \prod_{p\in Q_0} \ \Gr(e_p,m_p) \quad \longrightarrow \quad \Gr_{\norm\ue}(M)
\]
where $\norm\ue=\sum_{p\in Q_0} e_p$. This allows to define the \emph{Schubert decomposition of $\Gr_\ue(M)$ w.r.t.\ $\cB$} as the pull-back of the Schubert decomposition of $\Gr_{\norm\ue}(M)$.

We will explain this definition in more detail, assuming the following general hypothesis that will be valid throughout the paper unless explicitly mentioned otherwise. Let $Q$ be a quiver and $M$ a $Q$-module with basis $\cB$ (as a $k$-module). Then we assume the following property.
\begin{enumerate}
 \item[] The intersection $\cB_p=\cB\cap M_p$ is a basis of $M_p$ for every $p\in Q_0$. In other words, $\cB=\coprod_{p\in Q_0} \cB_p$.
\end{enumerate}

For a subset $\beta$ of $\cB$, we define $\beta_p=\beta\cap \cB_p$. The \emph{type of $\beta$} is the dimension vector $\ue=(e_p)$ with $e_p=\#\beta_p$. If $\cB$ is an ordered basis of $M$, $p\in Q_0$ and $\beta$ and $\gamma$ are subsets of $\cB$ of the same type $\ue$, then we write $\beta_p\preceq\gamma_p$ if we have $b_{p,l}\leq c_{p,l}$ for all $l\in\{1,\dotsc,e_p\}$ where we write $\beta_p=\{b_{p,1},\dotsc,b_{p,e_p}\}$ and $\gamma=\{c_{p,1},\dotsc,c_{p,{e_p}}\}$, ordered by size. We write $\beta\preceq\gamma$ if $\beta_p\preceq\gamma_p$ for all $p\in Q_0$. 

Let $\cB$ be a basis of $M$. The Pl\"ucker coordinates of the \emph{product Grassmannian}
\[
 \Gr_\ue(\um) \quad = \quad \prod_{p\in Q_0}\ \Gr(e_p,m_p) \quad \subset\quad \prod_{p\in Q_0}\ \P^{\binom{m_p}{e_p}-1}
\]
are the $e_p\times e_p$-minors 
\[
 \Delta_{\beta_p}: \quad v^p \quad \longmapsto \quad \det \ (v^p_{i,j})_{i\in\beta_p,j=1,\dotsc, e_p}
\]
of $m_p\times e_p$-matrices $v^p=(v^p_{i,j})_{i\in\cB_p,j=1,\dotsc,e_p}$ where $p$ varies through $Q_0$ and $\beta$ through the subsets of $\cB$ of type $\ue$. We denote by $U_\beta$ the canonical open subset of $\Gr_\ue(\um)$ with $\Delta_{\beta_p}=1$ for all $p\in Q_0$.

Let $\cB$ be an ordered basis of $M$ and $\beta\subset\cB$ a subset. Then define the \emph{Schubert cell $C_\beta(\um)$ of $\Gr_\ue(\um)$} as the intersection of $U_\beta$ with the vanishing set of all $e_p\times e_p$-minors $\Delta_{\gamma_p}$ with $\gamma_p\succeq\beta_p$, seen as a locally closed and reduced subscheme of $\Gr_\ue(\um)$. We define the \emph{Schubert cell $C_\beta^M$ of $\Gr_\ue(M)$} as the pull-back of $C_\beta(\um)$ along the closed embedding $\Gr_\ue(M)\hookrightarrow \Gr_\ue(\um)$. Then $C_\beta^M$ is a locally closed subscheme of $\Gr_\ue(M)$. Note that $C_\beta^M$ is in general not reduced (cf.\ Example \ref{ex: quiver with one vertex and one arrow}). Sometimes we refer to the reduced subscheme $C_\beta^{M,\red}=(C_\beta^M)^\red$ as a \emph{reduced Schubert cell}.

By the Schubert decomposition of usual Grassmannians, $\Gr_\ue(\um)$ decomposes into the Schubert cells $C_\beta(\um)$ where $\beta$ ranges through all subsets of $\cB$ of type $\ue$. The pull-back of this decomposition yields the \emph{decomposition}
\[
 \varphi: \quad \coprod_{\beta\subset\cB\text{ of type }\ue} C_\beta^M \quad \longrightarrow \quad \Gr_\ue(M),
\]
i.e.\ a morphism of $k$-schemes such that the restriction of $\varphi$ to one cell $C_\beta^M$ is a locally closed embedding and such that $\varphi$ induces a bijection between $K$-rational points for every field extension $K$ of $k$. We call this decomposition the \emph{Schubert decomposition of the quiver Grassmannian $\Gr_\ue(M)$ (w.r.t.\ $\cB$)}. In agreement with \cite{LL09}, we also write 
\[
 \Gr_\ue(M) \quad = \quad \decomp_{\beta\subset\cB\text{ of type }\ue} C_\beta^M
\]
for the Schubert decomposition. We use the modified symbol ``$\decomp$'' in order to avoid a confusion with the disjoint union of $k$-schemes.

Note that 
\[
 \Gr_\ue(M) \quad = \quad \decomp_{\beta\subset\cB\text{ of type }\ue} C_\beta^{M,\red}
\]
is also a decomposition of $\Gr_\ue(M)$, which we call the \emph{reduced Schubert decomposition of $\Gr_\ue(M)$}.

\begin{rem}\label{rem: irrelevance of oredering of basis elements of different vertices}
 The Schubert decomposition of $\Gr_\ue(M)$ depends only on the ordering of the subsets $\cB_p$ of $\cB$ and not on the ordering of elements $b\in\cB_p$ and $b'\in\cB_{p'}$ for different $p\neq p'$. However, we choose to endow $\cB$ with a linear order (and thus superfluous information at this point) since this is needed for the Schubert decomposition of a push-forward of $M$, cf.\ Section \ref{section: push-forwards}.
\end{rem}

\begin{rem}
 Note that the Schubert cells $C_\beta(\um)$ of $\Gr_\ue(\um)$ are affine spaces as products of Schubert cells of usual Grassmannians, but that the Schubert cells $C_\beta^M$ are in general not affine spaces. Since every projective $k$-scheme can be realized as a quiver Grassmannian, it is clear that this cannot be the case. Even if there exists a decomposition of $\Gr_\ue(M)$ into affine spaces for some choice of an ordered basis, a different choice of ordered basis might yield Schubert cells of a different shape (see Example \ref{ex: union of two projective lines}). In view towards Theorem \ref{thm: reineke}, I expect that every affine $k$-scheme of finite type can appear as a Schubert cell of a quiver Grassmannian for appropriate $Q$, $M$, $\ue$ and $\beta\subset\cB$.
\end{rem}

\subsection{Schubert decompositions for acyclic quiver}
\label{subsection: schubert decompositions for acyclic quiver}

In case the quiver $Q$ is acyclic, i.e.\ without oriented cycles, we find the following alternative description of the Schubert decomposition of $\Gr_\ue(M)$, cf.\ Section 6 in \cite{Caldero-Reineke08}. Let $H=k[Q]$ be the path algebra of $Q$ and $H^\ast$ be the unit subgroup. We can embed $H^\ast$ as a subgroup of $\GL(M,k)$. Since $Q$ is acyclic, $H^\ast$ is contained in a Borel subgroup $B$ of $\GL(M,k)$. 

The choice of a Borel subgroup $B$ of $\GL(M,k)$ is equivalent to the choice of an ordered basis $\cB$ for $M$ with the property that $B$ is the subgroup of upper triangular matrices in this basis. Note that in general, the basis $\cB$ does \emph{not} satisfy Hypothesis (H).

This choice defines a Schubert decomposition
\[
 \Gr_{\norm\ue}(M) \quad = \quad \decomp_{j} \ X_j
\]
of the usual Grassmannian $\Gr_{\norm\ue}(M)$ of submodules of rank $\norm\ue$ of $M$. As explained in \cite{Caldero-Reineke08}, the subscheme $\Gr_{\norm\ue}(M)^{H^\ast}$ of fixed points equals the disjoint union of all quiver Grassmannians $\Gr_{\ue'}(M)$ with $\norm{\ue'}=\norm{\ue}$. Therefore, we yield the decomposition
\[
 \coprod_{\norm{\ue'}=\norm\ue} \ \Gr_{\ue'}(M) \quad = \quad \decomp_j \ X_j^{H^\ast},
\]
which restricts to a decomposition $\Gr_\ue(M)=\decomp \bigl(X_j^{H^\ast}\cap\Gr_\ue(M)\bigr)$ into reduced subschemes.

This decomposition coincides with the (reduced) Schubert decomposition that we have defined in the previous section. In particular, if $\cB$ satisfies Hypothesis (H), the cells $X_j^{H\ast}\cap\Gr_\ue(M)$ coincide with the reduced cells $C_\beta^{M,\red}$. This means that the decomposition $\Gr_\ue(M)=\decomp \bigl(X_j^{H^\ast}\cap\Gr_\ue(M)\bigr)$ is the same as $\Gr_\ue(M)=\decomp C_\beta^M$.

\subsection{$K$-rational points of Schubert cells}
\label{subsection: K-rational points}

 Let $K$ be a ring extension of $k$. Using the canonical covering $\{U_\beta\}$ of $\Gr_\ue(\um)$, we can describe the $K$-rational points of a Schubert cell $C_\beta^M$ as follows. 

 A $K$-rational point of $C_\beta^M$ defines a subrepresentation $V$ of $M_K=M\otimes_kK$. This subrepresentation satisfies that for every $p\in Q_0$, the submodule $V_p$ of $M_p$ is generated by a set of vectors $v^p=(v_b)_{b\in \beta_p}$ where each $v_b$ is of the form 
\[
 v_b \quad = \quad 1 \cdot b \ + \ \sum_{\substack{b'\in\cB_p-\beta_p\\ b'<b}} v_{b',b} \cdot b'
\]
for some $v_{b',b}\in K$. Note that the coefficients $v_{b',b}$ are uniquely determined by $V$. This means that a $K$-rational point $V$ corresponds to a $\norm\cB\times\norm\beta$-matrix $v$ with coefficients $v_{b',b}\in K$, which satisfy $v_{b,b}=1$ and $v_{b',b}=0$ whenever $b'>b$ or $b'\in\beta$, but $b'\neq b$. In other words, $v$ is in row echelon form and all coefficients of a row containing a pivot $1$ are zero, except for the pivot itself.

Conversely, a choice of $v_{b',b}\in K$ yields a $K$-rational point $V$ of $C_\beta(\um)$, which, however, does not have to lie in $C_\beta^M(K)$. For certain cases of $M$, we will work out the conditions on the coefficients $v_{b',b}$ to come from a $K$-rational point $V$ of $C_\beta^M$ (see Section \ref{subsection: defining equations for schubert cells}).

\subsection{Examples}
\label{subsection: examples of schubert cells}

\begin{ex}[One point quiver and usual Grassmannians]\label{ex: quiver with one vertex}
 Let $Q$ be a quiver that consists of a single point and $M=k^m$ the $Q$-module with basis $\cB=(b_1,\dotsc, b_m)$. We consider the usual Schubert decomposition 
 \[
  \Gr(e,m) \quad = \quad \decomp_{1\leq i_1<\dotsb<i_e\leq m} \ C_{i_1,\dotsc,i_e}
 \]
 where $C_{i_1,\dotsc,i_e}$ is the reduced subscheme of $\Gr(e,m)$ with $K$-rational points
 \[
  \bigr\{ \ V\subset M_K \ \bigr| \ \text{for all }l=1,\dotsc,n\text{ and }k\text{ such that }i_k\leq l <i_{k+1}, \ \dim (V_l\cap N) = k \ \bigl\}
 \]
 for any field extension $K$ of $k$ where $V_l=\vspan\{b_1,\dotsc,b_l\}$. Then $C_{i_1,\dotsc,i_e}$ can be identified with $C_\beta^M$ for $\beta=\{b_{i_1},\dotsc,b_{i_e}\}$ if $\cB$ is ordered by $b_1<\dotsc <b_e$. This shows that we recover the Schubert decomposition of usual Grassmannians as a special case. 
\end{ex}

\begin{ex}[Flag varieties]\label{ex: flag varieties}
 The same is true for flag variety if we realize them as follows. Let $\ue=(e_1,\dotsc,e_r)$ be the type of the flag variety $X=X(e_1,\dotsc,e_r)$ of subspaces of $k^m$. Let $Q$ be the quiver
 \[ 
  \xymatrix{ 1\ar[r]^{\alpha_1} & 2 \ar[r]^{\alpha_2} & \dotsb \ar[r]^{\alpha_{r-1}} & r}
 \]
 and $M$ the $Q$-module $k^m\stackrel\id\longrightarrow \dotsb \stackrel\id\longrightarrow k^m$. Then $X$ is isomorphic to $\Gr_\ue(M)$. If we order the standard basis $\cB=\{b_{k,p}\,|\,k=1\dotsc,m;\,p=1,\dotsc,r\}$ of $M$ lexicographically, i.e. $b_{k,p}<b_{l,q}$ if $p<q$ or if $p=q$ and $k<l$, then the decomposition 
 \[
  \Gr_\ue(M) \quad = \quad \decomp_{\beta\subset\cB\text{ of type }\ue} C_\beta^M
 \]
 coincides with the usual decomposition of $X$ into Schubert cells.
\end{ex}

\begin{ex}[One loop quiver]\label{ex: quiver with one vertex and one arrow}
 Let $Q$ be the quiver with one vertex $p$ and one arrow $\alpha:p\to p$.  Let $M$ be the $Q$-module given by $M_p=k^m$ with the standard basis $\cB=(b_1,\dotsc,b_m)$ and by $M_\alpha=J(\lambda)$ where $J(\lambda)$ is a maximal Jordan block with $\lambda$ on the diagonal and $1$ on the upper side-diagonal. Let $e\leq m$. Considering $K$-rational points for a field extension $K$ of $k$, one sees easily that $\Gr_e(M,K)=\A^0(K)=C_\beta^M(K)$ with $\beta=\{b_1,\dotsc,b_e\}$, and $C_\gamma^M=\emptyset$ for other subsets $\gamma\subset\cB$ of cardinality $e$. 

 This means that the reduced cell $C_\beta^{m,\red}$ is isomorphic to $\A^0$. However, the non-reduced structure of $C_\beta^M$ is more involved. For our choice of ordering, it turns out that $C_\beta^M$ is indeed reduced, while $\Gr_\ue(M)$ is not. For another choice of ordering $C_\beta^M$ might be isomorphic to the non-reduced scheme $\Gr_\ue(M)$.

 We explain this in the example $m=2$ and $e=1$. The quiver Grassmannian $\Gr_\ue(M)$ is given as the vanishing set of the homogeneous equation 
 \[
  \det \ \tinymat{X}{\lambda X+Y}{Y}{\lambda Y} \quad = \quad \lambda XY \ - \ \lambda XY \ - \ Y^2 \quad = \quad Y^2
 \]
 (also cf.\ Example 2 in \cite{Cerulli11}). This means that $\Gr_\ue(M)=\Spec\bigl( k[\epsilon]/(\epsilon^2)\bigr)$ is non-reduced. However, the cell $C_\beta^M$ is defined by the open condition $X$ invertible and the closed condition $Y=0$. The latter equation forces $C_\beta^M$ to be $\Spec k=\A^0$.

 If we reverse the order of $\cB$, i.e. $b_2>b_1$, then the unique $K$-rational point of $\Gr_\ue(M)$ is still contained in the cell $C_{\beta}^M$, but $C_\beta^M$ is only defined by the open condition $X$ invertible. This means that $C_\beta^M\simeq \Gr_\ue(M)$ is a non-reduced scheme.
\end{ex}

\begin{ex}[Kronecker quiver]\label{ex: kronecker quiver}
 Let $Q$ be the Kronecker quiver with two vertices $1$ and $2$ and two arrows $\alpha,\beta:1\to 2$. A regular representations of $Q$ is a $Q$-module $M$ with $M_1=M_2=k^{n}$, $M_\alpha=\id$ and $M_\beta=J(\lambda)$ for some positive integer $n$ and some $\lambda\in k$. By Theorem 2.2 in \cite{Cerulli-Esposito11}, the quiver Grassmannian $\Gr_\ue(M)$ decomposes into affine spaces $X_L$, which coincide with the reduced Schubert cells $C_\beta^{M,\red}$ of $\Gr_\ue(M)$ w.r.t.\ the standard ordered basis of $M$. 
\end{ex}

\begin{ex} \label{ex: union of two projective lines}
 Consider the quiver $Q=\xymatrix{1\ar[r]^\alpha& 2}$ and the module $M$ that is described as follows. Let $M_1=M_2=k^2$, and let $\cB=\{b_1,\dotsc,b_4\}$ be the standard basis, i.e.\ $b_1=\binom 10$ and $b_2=\binom 01$ in $M_1$ and $b_3=\binom 10$ and $b_4=\binom 01$ in $M_2$. Let $M_\alpha:M_1\to M_2$ be the linear map that is described by the matrix $M_\alpha=\tinymat 1000$ in the bases $\cB_1$ and $\cB_2$. For the dimension vector $\ue=(1,1)$ and the ordering $b_1<b_2<b_3<b_4$ of $\cB$, we have the decomposition of $\Gr_\ue(M)$ into the four cells $C^M_{\{b_1,b_3\}}$, $C^M_{\{b_2,b_3\}}$, $C^M_{\{b_1,b_4\}}$ and $C^M_{\{b_2,b_4\}}$. 

 We use the notation $V=\tinybmat abcd$ for the submodule $V$ of $M_K$ with dimension vector $\ue=(1,1)$ and $V_1=\langle\binom ac\rangle$ and $V_2=\langle\binom bd\rangle$ where the coefficients lie in a field extension $K$ of $k$. Bearing the condition $M_\alpha(V_1)\subset V_2$ in mind, we find the following description for the $K$-rational points of the four cells:
 \begin{align*}
  C^M_{\{b_1,b_3\}}(K) \ &= \ \Bigl\{\ \tinybmat 1100 \ \Bigr\}                         && \simeq \A^0(K)\\
  C^M_{\{b_2,b_3\}}(K) \ &= \ \Bigl\{\ \tinybmat v110 \ \Bigr| \ v\in K \ \Bigr\}       && \simeq \A^1(K)\\
  C^M_{\{b_1,b_4\}}(K) \ &= \ \Bigl\{\ \tinybmat 1w01 \ \Bigr| \ \tinyvec 10\in\vspan\{\tinyvec w1\} \ \Bigr\}          && =\emptyset \\
  C^M_{\{b_2,b_4\}}(K) \ &= \ \Bigl\{\ \tinybmat vw11 \ \Bigr| \ v=0, w\in K \ \Bigr\}  && \simeq \A^1(K).
 \end{align*}
 Thus $\Gr_\ue(M)$ is isomorphic to two projective lines that intersect in one point.

 A reordering of the $b_1$ and $b_2$ is the same as reordering the rows of the matrix $\tinymat 1000$. If we calculate the Schubert cell of $\tinymat 0010:k^2\to k^2$ with the same ordering of the basis elements as above, we find that $C^M_{\{b_1,b_3\}}=\emptyset$, that $C^M_{\{b_2,b_3\}}=C^M_{\{b_1,b_4\}}=\A^0$ and that $C^M_{\{b_2,b_4\}}$ is isomorphic to two affine lines that intersect in one point. This shows that in general, it depends on the ordering of the basis $\cB$ whether the Schubert decomposition yields affine spaces as Schubert cells or not.
\end{ex}

\subsection{Disjoint unions of quivers}
\label{subsection: disjoint unions of quivers}

A subquiver $S$ of $Q$ is a quiver such that $S_0\subset Q_0$ and  $S_1\subset Q_1$, and such that the source and target maps of $S$ and $Q$ coincide. Let $M$ be a $Q$-module with basis $\cB$. Then the \emph{restriction $M_S$ of $M$ to a subquiver $S$ of $Q$} is the $S$-module with $M_{S,p}=M_p$ for $p\in S_0$ and $M_{S,\alpha}=M_\alpha$ for $\alpha$ in $S$. The set $\cB_S=\cB\cap M_S$ is a basis for $S$. The following fact is obvious, but useful.

\begin{lemma}\label{lemma: disjoint union}
 Let $Q=S\amalg T$ be the disjoint union of $T$ and $S$ and let $M$ be a $Q$-module with ordered basis $\cB$. Let $M_S$ and $M_T$ be the restrictions of $M$ to $S$ resp.\ $T$. Then $C_\beta^M=C_{\beta_S}^{M_S}\times C_{\beta_T}^{M_T}$ for every subset $\beta\subset B$ where $\beta_S=\beta\cap M_S$ and $\beta_T=\beta\cap M_T$. \qed
\end{lemma}

\begin{ex} \label{ex: unions of one vertex quivers}
 This yields a generalization of the previous examples. Namely, if $Q$ is a quiver and $M$ a representation such that the restriction of $M$ to each connected component $S$ of $Q$ is isomorphic to one of the $S$-modules of Examples \ref{ex: quiver with one vertex}--\ref{ex: union of two projective lines}, then there is an ordered basis $\cB$ of $M$ such that 
 \[
  \Gr_\ue(M) \quad = \quad \decomp_{\beta\subset\cB\text{ of type }\ue} C_\beta^M
 \]
 is a decomposition into affine spaces.
\end{ex}

\section{Tree extensions}
\label{section: tree extensions}

In this section, we investigate Schubert cell decompositions for trees. More precisely, we prove a relative theorem for tree extension $T$ of quivers $S$ that puts the Schubert cells of the tree extension into relation to the Schubert cells of $S$.

Let $T$ be a quiver with subquiver $S$. We denote by $T-S$ the subquiver that consists of all arrows of $T$ that are not in $S$ and all vertices that are not in $S$, or that are sources or targets of an arrow in $T-S$. Note that $S$ and $T-S$ can have vertices in common, but no edge. We denote by $T/S$ the quotient quiver, which is obtained from $T$ by removing all edges in $S$ and identifying all vertices of $S$. We say that $T$ is a \emph{tree extension of $S$} if $T/S$ is a tree (as a geometric graph). 

Let $M$ be a $T$-module with ordered basis $\cB$. We write $\beta<\beta'$ for two subsets $\beta$ and $\beta'$ of $\cB$ if $b<b'$ for all $b\in\beta$ and $b'\in\beta'$, and $\beta\leq\beta'$ if $(\beta-\beta')<(\beta\cap\beta')<(\beta'-\beta)$. We write $p<q$ for two vertices $p$ and $q$ of $T$ if $\cB_p<\cB_q$. Note that the relation $\beta\leq\beta'$ differs from the relation $\beta\preceq\beta'$ from Section \ref{section: schubert cells}. We say that an ordered basis $\cB$ \emph{induces an ordering of $T$} if for all distinct vertices $p$ and $q$ of $T$ either $p<q$ or $q<p$.

Let $S$ be a subquiver of $T$. We denote the restriction of $M$ to $S$ by $M_S$. Let $\cB_S$ a basis of $M_S$ and assume that $T$ is a tree extension of $S$. An \emph{extension of $\cB_S$ to $M$} is an ordered basis $\cB$ of $M$ whose intersection with $M_S$ is $\cB_S$ as ordered sets. An basis $\cB$ of $M$ is \emph{ordered above $S$} if $\cB_S\leq \cB$, if it induces an ordering of $T$, if $p_0<\dotsb<p_r$ for all paths $(p_0,\dotsc,p_r)$ with $p_0\in S_0\cap (T-S)_0$ and $p_1,\dotsc,p_r\in T_0-S_0$ pairwise distinct and if for all $\alpha\in (T-S)_1$, the linear map $M_\alpha$ is represented by the identity matrix w.r.t.\ the ordered bases $\cB_{s(\alpha)}$ and $\cB_{t(\alpha)}$.

\subsection{Results for tree extensions}
\label{subsection: results for tree extensions}

We will prove all results together at the end of this section.

\begin{lemma}\label{lemma: basis extensions for tree extensions}
 Let $T$ be a tree extension of $S$. Let $M$ be a $T$-module such that $M_\alpha$ is an isomorphism for all arrows $\alpha$ in $T-S$. Let $M_S$ be the restriction of $M$ to $S$ and $\cB_S$ an ordered basis of $M_S$ that induces an ordering of $S$. Then there exists an extension $\cB$ of $\cB_S$ that is ordered above $S$.
\end{lemma}

\begin{thm}\label{thm: tree extensions}
 Let $T$ be a tree extension of $S$. Let $M$ be a $T$-module and $M_S$ the restriction of $M$ to $S$. Let $\cB$ be an ordered basis of $M$ that is ordered above $\cB_S=\cB\cap M_S$. Let $\beta$ be a subset of $\cB$ and $\beta_S=\beta\cap\cB_S$. Then the following holds true.
 \begin{enumerate}
  \item The Schubert cell $C_\beta^M$ is empty if and only if $C_{\beta_S}^{M_S}$ is empty or if there exists an arrow $\alpha:p\to q$ in $T-S$ such that $M_\alpha(\beta_p)\not\subset\beta_q$. 
  \item If $C_{\beta}^{M}$ is not empty, then $C_\beta^M\simeq C_{\beta_S}^{M_S}\times\A^{n_\beta}$ for $\beta_S=\beta\cap M_S$ and some $n_\beta\geq0$.
 \end{enumerate}
\end{thm}

\begin{thm}\label{thm: Grassmannian fibrations}
 Let $T$ be a tree extension of $S$. Let $M$ be a $T$-module such that $M_\alpha$ is an isomorphism for all arrows $\alpha$ in $T-S$. Let $M_S$ be the restriction of $M$ to $S$. Let $\ue$ be a dimension vector for $T$ and $\ue_S$ the restriction of $\ue$ to $S$. Let $\kappa=\#T_0-\#S_0$. Then there is a sequence $S=T^{(0)}\subset T^{(1)} \subset\dotsb\subset T^{(\kappa)}=T$ of tree extensions of $S$ and a sequence
 \[
 \xymatrix{\Phi:\quad \Gr_\ue(M) \quad \ar@{->>}[r]^(0.43){\varphi_\kappa} & \quad \Gr_{\ue^{(\kappa-1)}}(M^{(\kappa-1)}) \quad \ar@{->>}[r]^(0.64){\varphi_{\kappa-1}} & \quad \dotsb \quad \ar@{->>}[r]^(0.28){\varphi_{1}} & \quad \Gr_{\ue^{(0)}}(M^{(0)}) \ = \ \Gr_{\ue_S}(M_S)}
 \]
 of fibre bundles $\varphi_i$ whose fibres are Grassmannians $\Gr(\tilde e_i,\tilde m_i)$ for certain integers $\tilde e_i\leq\tilde m_i$ and $i=1,\dotsc,\kappa$. Here $M^{(i)}$ and $\ue^{(i)}$ are the restrictions of $M$ resp.\ $\ue$ to $T^{(i)}$. 

 In particular, the morphism $\Phi:\Gr_\ue(M)\twoheadrightarrow\Gr_{\ue_S}(M_S)$ is smooth and projective.
\end{thm}

\begin{rem}
 Note that the sequence $S=T^{(0)}\subset T^{(1)} \subset\dotsb\subset T^{(\kappa)}=T$ and the corresponding quiver Grassmannians $\Gr_{\ue_{p}}(M_{p})$ are not unique, but depend on a choice of numbering of the vertices in $T_0-S_0$. However, the fibres $\Gr(\tilde e_p,\tilde m_p)$ are uniquely determined up to permutation of indices, and the morphism $\Phi:\Gr_\ue(M)\twoheadrightarrow\Gr_{\ue_S}(M_S)$ is canonical. In so far, Theorem \ref{thm: Grassmannian fibrations} can be seen as a Krull-Schmidt theorem for quiver Grassmannians of tree extensions.
\end{rem}

\subsection*{Proof of Lemma \ref{lemma: basis extensions for tree extensions}, Theorem \ref{thm: tree extensions} and Theorem \ref{thm: Grassmannian fibrations}}
\label{subsection: the proofs}

 All claims will be proven by an induction on $\kappa=\#(T_0-S_0)$. If $\kappa=0$, then $T=S$ and there is nothing to prove. This establishes the base case.

 If $\kappa>0$, then we choose an end of $T$ that does not lie in $S$, i.e.\ a vertex in $T_0-S_0$ that is connected to only one arrow $\alpha$. We consider the case that this vertex is the head of $\alpha$ separately from the case that it is the tail of $\alpha$.

 \bigskip\noindent
 \textbf{Case I:} There is an arrow $\alpha:p\to q$ such that $q$ is an end of $T$ that does not lie in $S$.

 \bigskip\noindent \emph{Proof of Lemma \ref{lemma: basis extensions for tree extensions}.} 
 Define $T'=T-\{q,\alpha\}$ and $M'$ as the restriction of $M$ to $T'$. By the induction hypothesis, there exists an ordered basis $\cB'$ of $T'$ that satisfies Lemma \ref{lemma: basis extensions for tree extensions}. We define $\cB_{q}:=M_{\alpha}(\cB'_{p})$ as an ordered set. Since $M_{\alpha}$ is an isomorphism, $\cB_{q}$ is a basis of $M_{q}$. We define $\cB=\cB'\cup\cB_{q}$ where the order of $\cB$ is defined such that $\cB'<\cB_{q}$. Then all claims of Lemma \ref{lemma: basis extensions for tree extensions} follow immediately.

 \bigskip\noindent \emph{Proof of Theorem \ref{thm: tree extensions}.} 
 We argue by considering $K$-rational points where $K$ is a ring extension of $k$. This will establish the statement $C_\beta^M\simeq C_{\beta_S}^{M_S}\times Z$ for a scheme $Z$ with $Z^\red=\A^{n_\beta}$. An additional argument will show that $Z$ is already reduced.

 Let $\beta\subset\cB$ and define $\beta'=\beta\cap M'$. Let $V$ be a $K$-rational point, i.e.\ a subrepresentation of $M_K=M\otimes_k K$. As explained in Section \ref{subsection: K-rational points}, $V$ can be identified with a $\norm\cB\times\norm\beta$-matrix in row echelon form with coefficients $v_{b',b}\in K$, pivots $v_{b,b}=1$ for $b\in\beta$ and $v_{b',b}=0$ if $b'>b$ or $b'\in\beta$, but $b'\neq b$.

 If we define $V_p=V\cap M_p$, then $M_\alpha(V_p)\subset V_q$ implies that pivots are mapped to pivots. Therefore if $C_\beta^M$ contains a $K$-rational point, then $M_\alpha(\beta_p)\subset\beta_q$, and the restriction $V'$ of $V$ to $T'$ is a $K$-rational point of $C_{\beta'}^{M'}$. Conversely, if $C_{\beta'}^{M'}$ contains a $K$-rational point $V'$ and  $M_\alpha(\beta_p)\subset\beta_q$, then the image $M_\alpha(V'_p)$ has generating vectors with pivots in $\beta_q$, and therefore $V'$ can be extended to a $K$-rational point of $C_\beta^M$. Since a scheme contains a $K$-rational point for some ring extension $K$ of $k$ if and only if the scheme is non-empty, this proves part \eqref{part1} of Theorem \ref{thm: tree extensions}.

 In case $C_\beta^M$ is non-empty, it contains a $K$-rational point $V$ for some ring extension $K$ of $k$. The columns of $V_q$ whose pivot corresponds to an element $M_\alpha(b)\in\beta_q$ for $b\in \beta_p$ are determined by the $b$-th column of $V_p$. All other columns can be chosen freely for $V$, which have 
 \[
  n_\beta' \quad = \quad \sum_{b\in\,(\beta_q-M_\alpha(\beta_p))} \ \#\,\{\, b'\in\cB_q \, | \, b'< b\text{ and }b'\notin\beta_q \,\}
 \]
 free coefficients. Since all equations are already defined over $k$, this establishes the isomorphism $C_\beta^{M}\simeq C_{\beta'}^{M'}\times Z$ with $Z^\red=\A^{n_\beta'}$.

 To see that the factor $\A^{n_\beta'}$ is reduced, recall that the defining equations for $V$ are linear in the coordinates of $V_q$. This is also true for the corresponding relations between the Pl\"ucker coordinates of $V$, cf.\ \cite[\S 9.1, Lemma 2]{Fulton97}. As a solution space of linear equations, the scheme $C_\beta^M$ is reduced. The finishes the proof of Theorem \ref{thm: tree extensions}.

 \bigskip\noindent \emph{Proof of Theorem \ref{thm: Grassmannian fibrations}.} 
 Let $\um$ be the dimension vector of $M$, $\ue\leq\um$ and $\ue'$ the restriction of $M$ to $T'$. We argue by considering $K$-rational points and prove that the natural morphism $\varphi:\Gr_\ue(M)\to \Gr_{\ue'}(M')$ is a fibre bundle with fibre $\Gr(\tilde e,\tilde m)$ for $\tilde e=e_q-e_p$ and $\tilde m=m_q-e_p$, up to a possible non-reduced structure sheaf of the fibre, which we will exclude by an additional argument. Using the induction hypothesis, this will establish the theorem.

 Note that in the case that $e_p>e_q$, we face the trivial case of an empty quiver Grassmannian $\Gr_{\ue}(M)$ and an empty fibre $\Gr(e_q-e_p,m_q-e_p)$. Thus we may assume that $e_q\geq e_p$. If $V$ is a $K$-rational point of $\Gr_{\ue}(M)$ for some ring extension $K$ of $k$, then $V_p$ determines an $e_p$-dimensional subspace of $V_q$ since $M_\alpha(V_p)\subset V_q$. This means that $V_q/M_\alpha(V_p)$, varies through $M_q/M_\alpha(V_p)$, which can be identified with a $K$-rational point of $\Gr(\tilde e,\tilde m)$. Therefore, the fibre $\varphi^{-1}(V)$ of every $K$-rational point $V$ of $\Gr_{\ue'}(M')$ is isomorphic to $\Gr(\tilde e,\tilde m)(K)$. 
 
 We will show that $\varphi(K)$ trivializes locally. To do so, we consider a $K$-rational point $V$ of $\Gr_{\ue}(M)$ and define $V'=\varphi(V)$. We choose a basis $\cB'$ of $M'$ and order it in such a way that $V'$ can be identified with an $\um\times\ue$-matrix in row echelon form that has pivots in the bottom rows $\beta_{p'}$ of the rows $\cB_{p'}$ for each vertex $p'$ of $T'$, i.e.\ such that $\beta_{p'}\geq\cB_{p'}$. Then the corresponding Schubert cell $C_{\beta'}^{M'}(K)$ for $\beta'=\bigcup_{p'\in T'}\beta_{p'}$ is an open neighborhood of $V'$ in $\Gr_{\ue'}(M')(K)$. 

 Further, we can assume that $\cB_p$ is ordered such that also $\beta_{q}\geq\cB_{q}$ if we extend $\cB'$ to a basis $\cB$ of $M$ by the rule $\cB_q=M_\alpha(\cB_p)$ and define $\beta_q$ as subset of $\cB_q$ that corresponds to rows that contain a pivot element of $V$. Then the Schubert cell $C_{\beta}^{M}(K)$ for $\beta=\beta'\cup\beta_q$ is an open neighborhood of $V$ in $\Gr_{\ue}(M)(K)$ and $\varphi$ restricts to a morphism $\varphi(K): C_{\beta}^{M}(K) \to C_{\beta'}^{M'}(K)$. 

 Since $\cB$ is an extension of $\cB'$ that is ordered above $T'$, we can apply Theorem \ref{thm: tree extensions} \eqref{part2} to obtain an isomorphism $C_{\beta}^{M}(K) \simeq C_{\beta'}^{M'}(K) \times \A^n(K)$ for some $n\geq0$. This shows that $\varphi$ is locally trivial, i.e.\ a fibre bundle.

 The fibre of $\varphi$ is reduced since it is given by a system of linear equations in the Pl\"ucker coordinates, cf.\ the proof of Theorem \ref{thm: tree extensions}. This finishes the proof of Theorem \ref{thm: Grassmannian fibrations}.

 \bigskip\noindent
 \textbf{Case II:} There is an arrow $\alpha:p\to q$ such that $p$ is an end of $T$ that does not lie in $S$.

 \bigskip\noindent \emph{Proof of Lemma \ref{lemma: basis extensions for tree extensions}.} 
 We proceed similar to Case I. We define $T'=T-\{p,\alpha\}$ and $M'$ as the restriction of $M$ to $T'$. By the induction hypothesis, there exists an ordered basis $\cB'$ of $T'$ that satisfies the lemma. We define $\cB_{q}:=M_{\alpha}^{-1}(\cB'_{p})$ as an ordered set. Note that $M_{\alpha}$ is an isomorphism, thus $\cB_{q}$ is a basis of $M_{q}$. We define $\cB=\cB'\cup\cB_{q}$ where the order of $\cB$ is defined such that $\cB'<\cB_{q}$. Then all claims of Lemma \ref{lemma: basis extensions for tree extensions} follow immediately.

 \bigskip\noindent \emph{Proof of Theorem \ref{thm: tree extensions}.} 
 If $C_\beta^M$ is non-empty, it contains a $K$-rational point $V$ for some ring extension $K$ of $k$. Then the restriction $V'$ of $V$ to $T'$ is a $K$-rational point of $C_{\beta'}^{M'}$, which shows that $C_{\beta'}^{M'}$ is non-empty. If $v$ is the matrix associated to $V$, then the condition $M_\alpha(V_p)\subset V_q$ shows that pivots are mapped to pivots, which means that $M_\alpha(\beta_p)\subset\beta_q$. Conversely, if $C_{\beta'}^{M'}$ contains a $K$-rational point $V'$ for some ring extension $K$ of $k$ and $M_\alpha(\beta_p)\subset\beta_q$, then we can extend $V'$ to a $T$-module $V$ by defining $V_p$ as follows: if $v'$ is the matrix associated to $V'$, then we define $V_p$ as the span of the column vectors of $v'$ that are labelled by those $b\in\beta_q$ that lie in the image $M_\alpha(\beta_p)$. This shows part \eqref{part1} of the theorem.

 Assume $C_\beta^M$ is non-empty, i.e.\ it contains a $K$-rational point $V$ with associated matrix $v$. For $b\in\beta_p$, the column vector $v_b$ of the submatrix $v^p$ of $v$ is determined by the column vector $v_{M_\alpha(b)}$ of $v^q$, up to adding a linear combination of the column vectors $v_{b'}$ of $v^q$ for which $b'\in \beta_q-M_\alpha(\beta_p)$ and $b'<M_\alpha(b)$. This yields
 \[
  n'_\beta \quad = \quad \sum_{b\in\beta_p} \ \# \{\ b'\in\beta_q \ | \ b'<M_\alpha(b)\text{ and } b'\notin M_\alpha(\beta_p) \ \}
 \]
 free coefficients. Therefore, $C_\beta^M\simeq C_{\beta'}^{M'}\times \A^{n'_\beta}$. Note that the factor $\A^{n'_\beta}$ is reduced for the same reason as explained in Case I. By the induction hypothesis, this establishes part \eqref{part2} of Theorem \ref{thm: tree extensions}.

 \bigskip\noindent \emph{Proof of Theorem \ref{thm: Grassmannian fibrations}.}
 The only difference to Case I is that $V_p$ varies while $V_q$ is fixed. Since $M_\alpha(V_p)\subset V_q$, this means that the $e_p$-dimensional $V_p$ varies in an $e_q$-dimensional space, i.e.\ the fibre of $\varphi:\Gr_\ue(M)\to\Gr_{\ue'}(M')$ is $\Gr(e_p,e_q)$. The rest of the proof is exactly as in Case I.

 \bigskip\noindent
 This finishes the proof of Lemma \ref{lemma: basis extensions for tree extensions}, Theorem \ref{thm: tree extensions} and Theorem \ref{thm: Grassmannian fibrations}.
\qed


\section{Push-forwards}
\label{section: push-forwards}

In this section, we generalize the results on Schubert cells for tree extensions to push-forwards along certain morphisms from tree extensions to other quivers.

A morphism $F:T\to Q$ of quivers is a map $F:T_0\cup T_1\to Q_0\cup Q_1$ such that $F(T_i)\subset Q_i$ for $i=0,1$ and such that for every arrow $\alpha$ in $T$, we have $F(s(\alpha))=s(F(\alpha))$ and $F(t(\alpha))=t(F(\alpha))$. We define the \emph{push-forward} $N=F_\ast M$ of a $T$-module $M$ as the $Q$-module with $N_{\tilde p}=\bigoplus_{p\in F^{-1}(\tilde p)} M_p $ for $\tilde p\in Q_0$ and 
\begin{equation}\label{eq: push-forward}
 N_{\tilde\alpha}: (n_p)_{p\in F^{-1}(\tilde p)} \ \mapsto \ (m_{q})_{q\in F^{-1}(\tilde q)} \qquad\text{with}\quad m_{q} \ = \ \sum_{\substack{\alpha\in F^{-1}({\tilde\alpha})\\ t(\alpha)=q}} M_\alpha(n_{s(\alpha)})
\end{equation}
for an arrow ${\tilde\alpha}:\tilde p\to \tilde q$ of $Q$. Note that a basis $\cB$ of $M$ is also a basis of $N=F_\ast M$.

A morphism $F:T\to Q$ of quivers is a \emph{winding} if for all arrows $\alpha\neq \alpha'$ of $T$ with $F(\alpha)=F(\alpha')$, also $s(\alpha)\neq s(\alpha')$ and $t(\alpha)\neq t(\alpha')$. Note that every inclusion of quivers is a winding and that windings are closed under compositions. Note further that the push-forward of a $T$-module $M$ along a winding $F:T\to Q$ satisfies that the sums defining the $m_{q}$ in Equation \eqref{eq: push-forward} range over at most $1$ element, and that every $n_p$ occurs in at most one of the sums defining the different $m_{q}$ for $q\in F^{-1}(\tilde q)$. In other words, $N_{\tilde\alpha}$ can be represented as a monomial block matrix whose non-zero blocks correspond to the $M_\alpha$ for $\alpha\in F^{-1}({\tilde\alpha})$.

\subsection{Defining equations for Schubert cells}
\label{subsection: defining equations for schubert cells}

In this section, we describe which matrices correspond to submodules of a push-forward $N=F_\ast M$ of a $T$-module $M$ along a winding $F:T\to Q$.

Let $\cB$ be an ordered basis of $M$ and $\beta\subset\cB$ a subset. Assume that $C_\beta^M\neq\emptyset$. For $q\in Q$, define $\cB_q=\bigcup_{p\in F^{-1}(q)} \cB_p$ and $\beta_q=\beta\cap\cB_q$. Let $K$ be a ring extension of $k$. By the observations of Section \ref{subsection: K-rational points}, every $K$-rational point $W$ of $C_\beta^{F_\ast M}$ can be identified with a matrix $w=(w_{b,b'})_{b\in\cB,b'\in\beta}$ with coefficients $w_{b,b'}\in K$ satisfying \emph{(a)} $w_{b,b'}=\delta_{b,b'}$ for $b,b'\in\beta$ where $\delta_{b,b'}$ is the Kronecker delta; \emph{(b)} $w_{b,b'}=0$ if $b\in\cB_{q}$ and $b'\in\beta_{q'}$ for distinct vertices $q$ and $q'$ of $Q$; and \emph{(c)} $w_{b,b'}=0$ if $b>b'$.

 Conversely, the column vectors of a matrix $w=(w_{b,b'})$ span a sub-$K$-module of $F_\ast M_K$, which is a sub-$Q$-module if and only if for every arrow $\tilde\alpha:\tilde p\to \tilde q$ in $Q$, we have $M_{\tilde\alpha}(W_{\tilde p})\subset W_{\tilde q}$. If we write $w_b$ for the $b$-th column vector of $w$, i.e. $w_b=(w_{b',b})_{b'\in\cB_{\tilde p}}$, then $M_{\tilde\alpha}(W_{\tilde p})\subset W_{\tilde q}$ if and only if for every $b\in\beta_{\tilde p}$, there are $\lambda_{b',b}\in K$ for $b'\in\beta_{\tilde q}$ such that 
 \begin{equation} \label{eq: condition for w being a subrepresentation1}
  M_{\tilde\alpha}(w_b) \quad = \quad \sum_{b'\in\beta_{\tilde q}} \ \lambda_{b',b} w_b'.
 \end{equation}

 We rewrite Equation \eqref{eq: condition for w being a subrepresentation1} as follows. Since $w_{b'',b'}=\delta_{b'',b'}$ for $b',b''\in\beta$, we conclude that $\lambda_{b',b}=\bigr(M_{\tilde\alpha}(w_b)\bigl)_{b'}$. Let $F^{-1}(\tilde\alpha)=\{\alpha_i:p_i \to q_i\}_{i=1,\dotsc,r}$ be the fibre of $F$ over $\tilde\alpha$. Define $w^{p',q'}$ as the submatrix $(w_{b,b'})_{b\in\cB_{p'},b'\in\beta_{q'}}$ of $w$ where $p'$ and $q'$ are vertices of $T$. Since $F$ is a winding, $M_{\tilde\alpha}$ decomposes into a direct sum of the $M_{\alpha_i}$ for $i=1,\dotsc,r$ and possibly a trivial morphism. Thus if $b\in\beta_{p_j}$ and $b'\in\beta_{q_i}$, then $\lambda_{b',b}=\bigl(M_{\tilde\alpha}(w_b)\bigr)_{b'}=\bigl(M_{\alpha_i}(w_{b'',b})_{b''\in\beta_{p_i}}\bigr)_{b'}$. If we define $M_{\alpha_i}\vert_{\beta_{q_i}}$ as the submatrix of $M_{\alpha_i}$ that contains only the $b$-th rows where $b$ is in $\beta_{q_i}$ (but all columns), then Equation \eqref{eq: condition for w being a subrepresentation1} (for varying 
$b\in \beta_{\tilde p}
$) can be expressed as
 \begin{equation} \tag*{{$E(\tilde\alpha,q_i,p_j)$}}\label{eq: condition for w being a subrepresentation2}
  M_{\alpha_i} \cdot w^{p_i,p_j} \quad = \quad \sum_{l=1}^r \ w^{q_i,q_l} \cdot  M_{\alpha_l}\vert_{\beta_{q_l}} \cdot  w^{p_l,p_j}
 \end{equation}
 for varying $i$ and $j$. 

Note that the equations of this system that correspond to rows $b\in\beta_{q_i}$ reduce to $M_{\alpha_i}(w_b)=M_{\alpha_i}(w_b)$, which is trivially satisfied. Therefore, only the equations for rows in $\cB_{q_i}-\beta_{q_i}$ yield proper conditions.

\subsection{Comparison of $C_\beta^M$ and $C_\beta^{F_\ast M}$} 
\label{subsection: comparison of M and F_ast M}

Let $F:T\to Q$ be a morphism of quivers and $M$ a $T$-module. Given an ordered basis $\cB$ of $M$ resp.\ $F_\ast M$ and a subset $\beta\subset\cB$, we like to compare the Schubert cells $C_\beta^M$ of $\Gr_\ue(M)$ and $C_\beta^{F_\ast M}$ of $\Gr_{F(\ue)}({F_\ast M})$ where $F(\ue)$ is the type of $\beta$ as a subset of ${F_\ast M}$, i.e.\ $F(\ue)=(f_{\tilde p})_{\tilde p\in Q_0}$ with $f_{\tilde p}=\#(\beta\cap {F_\ast M}_{\tilde p})=\sum_{p\in F^{-1}(\tilde p)}e_p$. 

There is a canonical closed embedding 
\[
 \iota_{F,\beta}^{M} \; : \quad C_\beta^M \quad \longrightarrow \quad C_\beta^{F_\ast M}
\]
 by sending a submodule $V$ of $M$ to the submodule $F_\ast V$ of $F_\ast M$. If $V$ is represented by the matrix $v$, then $F_\ast V$ is represented by the same matrix $v$. This defines a canonical closed embedding 
\[
 \iota_\ue^M\; : \quad \Gr_\ue(M) \quad \longrightarrow \quad \Gr_{F(\ue)}(F_\ast M). 
\]

Under a certain assumption on $F:T\to Q$, there exists a retraction to $\iota_{F,\beta}^{M}$. Namely, a morphism $F:T\to Q$ is called \emph{strictly ordered (w.r.t.\ $\cB$)} if for each pair of distinct arrows $\alpha:p\to q$ and $\alpha':p'\to q'$ of $T$ with $F(\alpha)=F(\alpha')$, we have that either $p<p'$ and $q<q'$ or $p>p'$ and $q>q'$. In other words, the ordering of $\cB$ defines a natural ordering of the arrows in the fibre $F^{-1}(\tilde\alpha)$ for every arrow $\tilde\alpha$ of $Q$. Note that every strictly ordered morphism is a winding.

For a strictly ordered winding $F:T\to Q$, we can define a morphism $\pi^M_{F,\beta}:C_\beta^{F_\ast M}\to C_\beta^M$ as follows. Let $W$ be a $K$-rational point where $K$ is a ring extension of $k$ and let $w$ be the associated $\norm\cB\times\norm\beta$-matrix with coefficients in $K$. We regard $w$ as a block matrix $(w^{p,p'})_{p,p'\in T_0}$ where $w^{p,p'}$ is the submatrix of $w$ whose rows are labelled by elements of $\cB_p$ and whose columns are labelled by elements of $\beta_{p'}$. Then $w^{p,p'}$ is the zero matrix if $F(p)\neq F(p')$ or if $p'<p$. If $\tilde\alpha:\tilde p\to \tilde q$ is an arrow in $Q$ and $F^{-1}(\tilde\alpha)=\{\alpha_i:p_i \to q_i\}_{i=1,\dotsc,r}$, then the submatrices $w^{p_i,p_j}$ satisfy Equation \ref{eq: condition for w being a subrepresentation2}. Since $F$ is strictly ordered, this reduces to 
\[
  M_{\alpha_i} \cdot w^{p_i,p_i} \quad = \quad w^{q_i,q_i} \cdot  M_{\alpha_i}\vert_{\beta_{q_i}} \cdot  w^{p_i,p_i}
\]
in case that $i=j$. This means that also the block matrix $v=(v ^{p,p'})_{p,p'\in T_0}$ with $v^{p,p}=w^{p,p}$ and $v^{p,p'}=0$ if $p\neq p'$ satisfies the Equation \ref{eq: condition for w being a subrepresentation2} for all choices of $\tilde\alpha$, $i$ and $j$. Therefore $v$ is associated to a $K$-rational point $V'$ of $C_\beta^{F_\ast M}$, which is the image $\iota_{F,\beta}^{M}(V)$ of the $K$-rational $V$ of $C_\beta^M$. This defines the morphism 
\[
 \pi^M_{F,\beta}\; :\quad C_\beta^{F_\ast M} \quad \longrightarrow \quad C_\beta^M,
\]
which is a retract to the embedding $\iota_{F,\beta}^{M}:C_\beta^M\to C_\beta^{F_\ast M}$. 

We summarize the facts of this section in the following proposition.

\begin{prop}\label{prop: inclusion of cells and the retract}
 Let $F:T\to Q$ be a morphism of quivers and $M$ a $T$-module with ordered basis $\cB$. For every subset $\beta\subset\cB$ of type $\ue$, there is a closed embedding $\iota_{F,\beta}^{M}:C_\beta^M\to C_\beta^{F_\ast M}$, which is the restriction of a closed embedding $\iota_\ue^M: \Gr_\ue(M) \to\Gr_{F(\ue)}(F_\ast M)$.

 If $F:T\to Q$ is a strictly ordered winding, then $\iota_{F,\beta}^{M}$ has a retract $\pi^M_{F,\beta}:C_\beta^{F_\ast M}\to C_\beta^M$.
 \qed
\end{prop}

\subsection{Relevant pairs and relevant triples}
\label{subsection: relevant pairs}

In this section, we introduce relevant pairs, which index the submatrices $w^{p,q}$ of a $K$-rational point $w$ that contain (possibly) non-zero variables. Further, we introduce relevant triples, which index the equations between the $w^{p,q}$ that define a Schubert cell.

Let $\cB$ be an ordered basis of $M$ that is ordered above $S$. Assume further that $M_\alpha$ is the identity matrix for all arrows $\alpha$ in $T-S$. Assume that $F:T\to Q$ is strictly ordered w.r.t.\ this ordering of $T$. We define the following functions on the set
\[
 \Adm^2 \ = \ \{ \ (p,p') \in T \ | \ F(p)=F(p'),\, p\leq p'\text{ and } p'\notin S \ \}
\]
of \emph{relevant pairs (on $T$ w.r.t.\ $S\subset T$ and $F:T\to Q$)}. The importance of relevant pairs is the following: fix a point $w_S$ of $C_\beta^{F_\ast M_S}$; if we want to extend this to a point $w$ of $C_\beta^M$, then we have to consider the various equations $E(\tilde\alpha,q_i,p_j)$ for block matrices $w^{p,p'}$ with $p,p'\in T$. If $p,p'\in S$, then $w^{p,p'}$ is determined by $w_S$. If $p>p'$, then $w^{p,p'}=0$. Therefore we have to inspect only those $w^{p,p'}$ for which $(p,p')$ is an relevant pair.

Let $d(p,S)$ be the distance from $p$ to $S$, i.e.\ the length of a shortest path from $p$ to a vertex in $S$. Then we define the \emph{distance of an relevant pair $(p,p')$ to $S$} as the number
\[
 \delta(p,p') \ = \ \max \ \{ \ d(p,S) \,,\, d(p',S) \ \}.
\]
We define the \emph{fibre length of an relevant pair $(p,p')$} as the number
\[
 \epsilon(p,p') \ = \ \#\{ \ p''\in T_0 \ | \ F(p'')=F(p)\text{ and }p\leq p'' < p' \ \}.
\]
The function
\[
 \Psi(p,p') \ = \ \bigl( \, \epsilon(p,p'), \delta(p,p'), p' \,)
\]
defines an embedding $\Psi:\Adm^2\to\N\times\N\times T_0$, which we order lexicographically, i.e. $(\epsilon,\delta,p)\leq(\epsilon',\delta',p')$ if $\epsilon<\epsilon'$, or $\epsilon=\epsilon'$ and $\delta<\delta'$, or $\epsilon=\epsilon'$, $\delta=\delta'$ and $p\leq p'$. This defines an ordering on the set of relevant pairs resp.\ on values $\Psi(p,p')$. 

Note that $\delta$, $\epsilon$ and $\Psi$ extend to functions on all pairs $(p,p')$ with $F(p)=F(p')$ and $p\leq p'$. Note further that $\delta(p,p')=0$ if and only if $p,p'\in S_0$. We extend the ordering of the $\Psi(p,p')$ for relevant pairs $(p,p')$ to all pairs $(p,p')$ with $F(p)=F(p')$ and $p\leq p'$ by the following rules: $\Psi(p,p')<\Psi(q,q')$ if $(q,q')$ is an relevant pair and $(p,p')$ is not; pairs that are not relevant are ordered lexicographically.

We define the set of \emph{relevant triples} as 
\[
 \Adm^3 \ = \ \{ (\tilde\alpha,q,p)\in Q_1\times T_0\times T_0 \ | \ q\in F^{-1}(t(\tilde\alpha))\text{ and }p\in F^{-1}(s(\tilde\alpha)) \ \}.
\]

Every relevant triple $(\tilde\alpha,t,s)$ leads to the equation $E(\tilde\alpha,t,s)$:
 \begin{equation*} 
  \sum_{\substack{\alpha\in F^{-1}(\tilde\alpha)\\ s(\alpha)\leq s, \, t\leq t(\alpha)}} \ w^{t,t(\alpha)}  M_{\alpha'}\vert_{\beta_{t(\alpha)}}  w^{s(\alpha),s} \quad = \quad \begin{cases} M_{\alpha'}  w^{s',s} & \text{if there is an }\alpha':s'\to t\text{ in }F^{-1}(\tilde \alpha), \\ 0 & \text{otherwise.}\end{cases}
 \end{equation*}
Note that the terms $M_{\alpha'} w^{s',s}$ and $w^{t,t''} M_{\alpha''}\vert_{\beta_{t''}} w^{s,s}$ (where $\alpha'':s\to t''$ is in $F^{-1}(\tilde\alpha)$) are of particular importance for us since they are linear in $w^{s',s}$ resp.\ (partly) linear in $w^{t,t''}$.

\subsection{Triple types}
\label{subsection: fibre types}

In this section, we describe different types of relevant triples $(\tilde\alpha,t,s)$ with respect to the shape of the equation $E(\tilde\alpha,t,s)$.

\subsection*{Type 0} There is no $\alpha\in F^{-1}(\tilde\alpha)$ with $s(\alpha)\leq s$ and $t\leq t(\alpha)$.
\[
 \xymatrix@R=0,5pc@C=6pc{  \ar@{-->}[r]|{}^{\text{no arrow}}& t \\  \ar@{-->}[r]|{}^{\text{no arrow}}& \\ s \ar@{-->}[r]|{}^{\text{no arrow}}& & T \ar[dd]^F \\ \\ s(\tilde\alpha) \ar[r]^{\tilde\alpha} & t(\tilde\alpha) & Q }
\]
In this case Equation $E(\tilde\alpha,t,s)$ is trivial, i.e. $0=0$.

\subsection*{Type 1} There is an arrow $\alpha:s\to t$ in $F^{-1}(\tilde\alpha)$.
\[
 \xymatrix@R=0,5pc@C=6pc{ s \ar[r]|{}^{a}& t & T \ar[dd]^F \\ \\ s(\tilde\alpha) \ar[r]^{\tilde\alpha} & t(\tilde\alpha) & Q }
\]
In this case Equation $E(\tilde\alpha,t,s)$ is 
\[
 M_{\alpha} w^{s,s} \ = \ w^{t,t} M_{\alpha}\vert_{\beta_{t}} w^{s,s}.
\] 
Note that $\epsilon(s,s)=\epsilon(t,t)=0$.

\subsection*{Type 2} There are arrows $\alpha':s'\to t$ and $\alpha'':s\to t''$ in $F^{-1}(\tilde\alpha)$ with $s'<s$ and $t'' < t$.
\[
 \xymatrix@R=0,5pc@C=6pc{ s' \ar[r]|{}^{\alpha'}& t \\ s(\alpha) \ar@{-->}[r]^{\alpha}& t(\alpha) \\ s \ar[r]^{\alpha''}& t'' & T \ar[dd]^F \\ \\ s(\tilde\alpha) \ar[r]^{\tilde\alpha} & t(\tilde\alpha) & Q }
\]
In this case Equation $E(\tilde\alpha,t,s)$ is 
\[
 M_{\alpha'} w^{s',s} \ - \ w^{t,t''} M_{\alpha''}\vert_{\beta_{t''}} w^{s,s} \ = \  \sum_{\alpha'\leq\alpha<\alpha''} \ w^{t,t(\alpha)}  M_{\alpha'}\vert_{\beta_{t(\alpha)}}  w^{s(\alpha),s}
\] 
where $\alpha'\leq\alpha$ means that $F(\alpha')=F(\alpha)$ and $s(\alpha')\leq s(\alpha)$ resp.\ $t(\alpha')\leq t(\alpha)$. Note that consequently $\Psi(s(\alpha),s)<\Psi(s',s)$ and $\Psi(t,t(\alpha))<\Psi(t,t'')$ for all $\alpha$ with $\alpha'<\alpha<\alpha''$.

We subdivide relevant triples of Type 2 into the following subtypes.
\begin{itemize}
 \item[\bf Type 2a.] $\Psi(t,t'') < \Psi(s',s)$.
 \item[\bf Type 2b.] $\Psi(s',s) < \Psi(t,t'')$.
\end{itemize}

\subsection*{Type 3} There is an arrow $\alpha'':s\to t''$ in $F^{-1}(\tilde\alpha)$ with $t<t''$, but there is no arrow $\alpha'\in F^{-1}(\tilde\alpha)$ with $t(\alpha')=t$.
\[
 \xymatrix@R=0,5pc@C=6pc{ \ar@{-->}[r]^{\text{no arrow}}& t \\ s(\alpha) \ar@{-->}[r]^{\alpha}& t(\alpha) \\ s \ar[r]^{\alpha''}& t'' & T \ar[dd]^F \\ \\ s(\tilde\alpha) \ar[r]^{\tilde\alpha} & t(\tilde\alpha) & Q }
\]
In this case Equation $E(\tilde\alpha,t,s)$ is 
\[
 w^{t,t''} M_{\alpha''}\vert_{\beta_{t''}} w^{s,s} \ = \  - \!\!\! \sum_{\substack{\alpha\in F^{-1}(\tilde\alpha)\\t< t(\alpha)<t''}} \ w^{t,t(\alpha)}  M_{\alpha'}\vert_{\beta_{t(\alpha)}}  w^{s(\alpha),s}
\] 
Note that $\Psi(t(\alpha),t)<\Psi(t,t'')$ for all $\alpha\in F^{-1}(\tilde\alpha)$ with $t<t(\alpha)<t''$. We subdivide relevant triples of Type 3 into the following subtypes.
\begin{itemize}
 \item[\bf Type 3a.] For all arrows $\alpha\in F^{-1}(\tilde\alpha)$ with $t<t(\alpha)<t''$, we have $\Psi(s(\alpha),s)<\Psi(t,t'')$.
 \item[\bf Type 3b.] There is an arrow $\alpha\in F^{-1}(\tilde\alpha)$ with $t<t(\alpha)<t''$ and $\Psi(s(\alpha),s)>\Psi(t,t'')$.
\end{itemize}

\subsection*{Type 4} There is an arrow $\alpha':s'\to t$ in $F^{-1}(\tilde\alpha)$ with $s'<s$, but there is no arrow $\alpha''\in F^{-1}(\tilde\alpha)$ with $s(\alpha'')=s$.
\[
 \xymatrix@R=0,5pc@C=6pc{ s' \ar[r]|{}^{\alpha'}& t \\ s(\alpha) \ar@{-->}[r]^{\alpha}& t(\alpha) \\ s \ar@{-->}[r]^{\text{no arrow}}& & T \ar[dd]^F \\ \\ s(\tilde\alpha) \ar[r]^{\tilde\alpha} & t(\tilde\alpha) & Q }
\]
In this case Equation $E(\tilde\alpha,t,s)$ is 
\[
 M_{\alpha'} w^{s',s} \ = \  \sum_{\substack{\alpha\in F^{-1}(\tilde\alpha)\\s'< s(\alpha)<s}} \ w^{t,t(\alpha)}  M_{\alpha'}\vert_{\beta_{t(\alpha)}}  w^{s(\alpha),s}
\] 
Note that $\Psi(s',s(\alpha))<\Psi(s',s)$ for all $\alpha\in F^{-1}(\tilde\alpha)$ with $s'<s(\alpha)<s$. We subdivide relevant triples of Type 4 into the following subtypes.
\begin{itemize}
 \item[\bf Type 4a.] For all arrows $\alpha\in F^{-1}(\tilde\alpha)$ with $s'<s(\alpha)<s$, we have $\Psi(t,t(\alpha))<\Psi(s',s)$.
 \item[\bf Type 4b.] There is an arrow $\alpha\in F^{-1}(\tilde\alpha)$ with $s'<s(\alpha)<s$ and $\Psi(t,t(\alpha))>\Psi(s',s)$.
\end{itemize}

\subsection*{Type 5} There is no arrow $\alpha\in F^{-1}(\tilde\alpha)$ with $s(\alpha)=s$ or $t(\alpha)=t$, but there is an arrow $\alpha\in F^{-1}(\tilde\alpha)$ with $s(\alpha)<s$ and $t<t(\alpha)$
\[
 \xymatrix@R=0,5pc@C=6pc{ \ar@{-->}[r]|{}^{\text{no arrow}}& t \\ s(\alpha) \ar[r]^{\alpha}& t(\alpha) \\ s \ar@{-->}[r]^{\text{no arrow}}& & T \ar[dd]^F \\ \\ s(\tilde\alpha) \ar[r]^{\tilde\alpha} & t(\tilde\alpha) & Q }
\]
In this case Equation $E(\tilde\alpha,t,s)$ is 
\[
 \sum_{\substack{\alpha\in F^{-1}(\tilde\alpha)\\s(\alpha)<s, \, t< t(\alpha)}} \ w^{t,t(\alpha)}  M_{\alpha'}\vert_{\beta_{t(\alpha)}}  w^{s(\alpha),s}  \ =  \ 0.
\]

\subsection{The main theorem}
\label{subsection: push-forward theorem}

In the context of a tree extension $T$ of $S$, a $T$-module $M$ with ordered basis $\cB$ that is ordered above $S$ and a morphism $F:T\to Q$, we will formulate the following hypothesis. Denote by $I^s(\tilde\alpha)$ the set of $p\in F^{-1}(s(\tilde\alpha))$ such that there is no $\alpha\in F^{-1}(\tilde\alpha)$ with $p=s(\alpha)$. Denote by $I^t(\tilde\alpha)$ the set of $p\in F^{-1}(t(\tilde\alpha))$ such that there is no $\alpha\in F^{-1}(\tilde\alpha)$ with $p=t(\alpha)$.

\subsection*{Hypothesis (H)} The morphism $F:T\to Q$ is strictly ordered. It satisfies the following condition for every relevant pair $(p,p')\in\Adm^2$ and $\tilde p=F(p)$:
\begin{itemize}
 \item for all $\tilde\alpha:\tilde p \to \tilde q$ in $Q$ with $p'\in I^s(\tilde\alpha)$ and for all $q\in F^{-1}(\tilde q)$, the triple $(\tilde\alpha,q,p')$ is of type 0;
 \item for all other $\tilde\alpha:\tilde p \to \tilde q$ in $Q$, there is an arrow $\alpha:p\to q$ in $F^{-1}(\tilde\alpha)$ such that $(\tilde\alpha,q,p')$ is of type 1 or 2b;
 \item for all $\tilde\alpha:\tilde q \to \tilde p$ in $Q$ with $p\in I^t(\tilde\alpha)$ and for all $q'\in F^{-1}(\tilde q)$, the triple $(\tilde\alpha,p,q')$ is of type 0;
 \item for all other $\tilde\alpha:\tilde q \to \tilde p$ in $Q$, there is an arrow $\alpha:q'\to p'$ in $F^{-1}(\tilde\alpha)$ such that $(\tilde\alpha,p,q')$ is of type 1 or 2a;
\end{itemize}
with one of the following two possible exceptions:
\begin{enumerate}
 \item there is one arrow $\tilde\alpha:\tilde p \to \tilde q$ and an arrow $\alpha:p\to q$ in $F^{-1}(\tilde\alpha)$ such that $(\tilde\alpha,q,p')$ is of type 2a or 4a; if $\alpha\in S_1$, then $M_\alpha$ is the identity matrix; or
 \item there is one arrow $\tilde\alpha:\tilde q \to \tilde p$ and an arrow $\alpha':q'\to p'$ in $F^{-1}(\tilde\alpha)$ such that $(\tilde\alpha,p,q')$ is of type 2b or 3a.
\end{enumerate}

\begin{thm}\label{thm: push-forward}
 Let $T$ be a tree extension of $S$ and $M$ a $T$-module with ordered basis $\cB$ that is ordered above $S$. Let $M_S$ be the restriction of $M$ to $S$ and $\cB_S=\cB\cap M_S$. Let $F:T\to Q$ be a morphism that satisfies Hypothesis (H). Let $n_\beta\geq0$ be the integer such that $C_\beta^M\simeq C_{\beta_S}^{M_S}\times \A^{n_\beta}$ where $\beta_S=\beta\cap\cB_S$ (cf.\ Theorem \ref{thm: tree extensions}).
 
 Then there is an integer $n_{F,\beta}\geq 0$ and an isomorphism $C_\beta^{F_\ast M}\simeq C_{\beta_S}^{F_\ast M_S}\times \A^{n_{\beta}}\times \A^{n_{F,\beta}}$ such that 
  \[
   \xymatrix@R=3pc{C_\beta^M \ar[rr]^(0.45)\sim \arincl[d]_{\iota_{F,\beta}^M}  && C_{\beta_S}^{M_S}\times\A^{n_{\beta}} \arincl[d]^{\bigl(\iota_{F,\beta_S}^{M_S},\id\bigr)} \\ C_\beta^{F_\ast M} \ar[rr]^(0.35)\sim && C_{\beta_S}^{F_\ast M_S}\times\A^{n_{F,\beta}}\times\A^{n_{\beta}}}
  \]
  commutes. In particular, $C_\beta^M$ is empty if and only if $C_\beta^{F_\ast M}$ is empty.
\end{thm}

\begin{proof}
 To begin with, note that $C_\beta^M$ is empty if and only if $C_\beta^{F_\ast M}$ is empty since $F:T\to Q$ is strictly ordered, and we have a retraction $\pi_{F,\beta}^M$ to $\iota^m_{F,\beta}:C_\beta^M\to C_\beta^{F_\ast M}$ (see Proposition \ref{prop: inclusion of cells and the retract}). Thus if $C_\beta^M$ is empty, the statement of the theorem is trivial, and we can assume that both $C_\beta^M$ and $C_\beta^{F_\ast M}$ are not empty.
 
 As we remarked before, a $K$-rational point $W$ of $C_\beta^{F_\ast M}$ with associated matrix $w=(w_{b,b'})$ is determined by the choice of a $K$-rational point $W_S$ of $C_\beta^{F_\ast M_S}$ (which might be thought as the restriction of $w$ to $S$) and by the coefficients of the submatrices $w^{p,p'}$ of $w$ where $(p,p')$ ranges through $\Adm^2$. The submatrices $w^{p,p'}$ are subject to the equations $E(\tilde\alpha,t,s)$ for relevant triples $(\tilde\alpha,t,s)$.

 We prove by induction over $\Psi\in\N\times\N\times T_0$ (with $\Psi\geq (1,0,p)$ for some $p\in T$) that the possible solutions for $w$ in the coefficients $w^{p,p'}$ with $\Psi(p,p')\leq\Psi$ satisfy the claim of the theorem for some $n_{\Psi,\beta}$ in place of $N_{F,\beta}$. 

 We establish base case $\Psi=(1,0,p)$ (for some $p\in T$). An relevant pair ${p,p'}$ satisfies $\Psi(p,p')\leq \Psi$ if and only if $p=p'$. This means that we have to find to a given $W_S$ solutions in the submatrices $w^{p,p}$ with $p\in T_0-S_0$. But this is the situation of Theorem \ref{thm: tree extensions} for $S$ replaced by $F(S)$ and $M_S$ replaced by $F_\ast M_S$. Thus the claim of the theorem for $\Psi$ with $n_{\Psi,\beta}=0$.

 Consider an relevant pair $(p,p')$ with $p\neq p'$. We will deduce the claim of the theorem for $\Psi(p,p')$ by use of the inductive hypothesis. To find the solutions in $w^{p,p'}$ depending on the $w^{q,q'}$ with $\Psi(q,q')<\Psi(p,p')$, we have to consider all those equations $E(\tilde\alpha,t,s)$ in which $w^{p,p'}$ appears as the submatrix with the largest value $\Psi(p,p')$. By Hypothesis (H), all relevant triples $(\tilde\alpha,t,s)$ are of type 0, 1, 2a, 2b, 3a or 4a. Note that for these types either $E(\tilde\alpha,t,s)$ is trivial (type 0) or the term of $E(\tilde\alpha,t,s)$ with the largest relevant pair $(q,q')$ appearing as an index is either $M_{\alpha'} w^{s',s}$ (\emph{case (a)}) or $w^{t,t''} M_{\alpha''}\vert_{\beta_{t''}} w^{s,s}$ (\emph{case (b)}). In case (a), $(s,s')=(q',q)$ and there is an arrow $\alpha':s'\to t$ in $F^{-1}(\tilde\alpha)$. In case (b), $(t,t'')=(q,q'))$ and there is an arrow $\alpha'':s\to t''$ in $F^{-1}(\tilde\alpha)$. 

 Let $E(\tilde\alpha,t,s)$ be an equation in which $(p,p')$ appears as the largest index. Consider $(\tilde\alpha,t,s)$ of type 2, i.e.\
 \begin{itemize}
  \item[(a)] there are arrows $\alpha:p\to q$ and $\alpha':p'\to q'$ in $F^{-1}(\tilde\alpha)$, or  
  \item[(b)] there are arrows $\alpha:q\to p$ and $\alpha':q'\to p'$ in $F^{-1}(\tilde\alpha)$. 
 \end{itemize}
 Since $(p,p')$ is relevant and since it is the index with the largest value $\Psi(p,p')$, the triple $(\tilde\alpha,t,s)$ is of type 2a in case {(a)} and of type 2b in case {(b)}. This means that none of the ``non-exceptional'' cases of Hypothesis (H) lead to an equation in which $(p,p')$ appears as its largest index.

 In the exceptional cases \eqref{part1} and \eqref{part2} of Hypothesis (H), we face indeed equations $E(\tilde\alpha,t,s)$ in which ${p,p'}$ as the largest index. Before we proceed inspecting the different types of exceptions, we will explain how to solve Equation $E(\tilde\alpha,t,s)$ in $w^{s',s}$ (case (a)) resp.\ $w^{t,t''}$ (case (b)).

 We inspect an relevant triple $(\tilde\alpha,t,s)$ in case (a), i.e.\ $(s',s)=(p,p')$ is the largest index that occurs in $E(\tilde\alpha,t,s)$ and there is an arrow $\alpha:p\to q$ in $F^{-1}(\tilde\alpha)$ with $q=t$. Then $(p,p')$ occurs twice in $E(\tilde\alpha,t,s)$; namely, in the term $M_\alpha w^{p,p'}$ and in the term $w^{q,q} M_{\alpha}\vert_{\beta_q}w^{p,p'}$. We will see that in all possible cases that we have to take into account, $M_\alpha$ is the identity matrix. Therefore the terms in question reduce to $w^{p,p'}$ and $w^{q,q}w^{p,p'}\vert_{\beta_q}$. Note that if the identity matrix $M_\alpha$ does not map $\beta_p$ into $\beta_q$, then $C_\beta^M$ and $C_\beta^{F_\ast M}$ are empty. Therefore, we can assume that all coordinates of $w^{p,p'}\vert_{\cB_q-\beta_q}$ are free. This means that we can make an arbitrary choice for the non-zero coordinates of $w^{p,p'}\vert_{\beta_q}$ and solve $E(\tilde\alpha,t,s)$ in $w^{p,p'}\vert_{\cB_q-\beta_q}$. The solution space is therefore an affine 
space. 

 Since we can write a special solution in $w^{p,p'}$ in terms of polynomials in the coefficients of the other matrices $w^{q,q'}$ appearing in $E(\tilde\alpha,t,s)$, it is clear that the fibration that is given by attaching the solution space in the coefficients of $w^{p,p'}$ to a set of coordinates for $w^{q,q'}$ with $\Psi(q,q')<\Psi(p,p')$ is a trivial vector bundle, i.e.\ there is an $n_{\Psi(p,p',\beta}\geq 0$ such that the solution space in all $w^{q,q'}$ with $\Psi(q,q')\leq\Psi(p,p')$ equals the solution space in all $w^{q,q'}$ with $\Psi(q,q')<\Psi(p,p')$ times $\A^{n_{\Psi(p,p'),\beta}}$. This also shows that the diagram of the theorem, restricted to the coordinates of the $w^{q,q'}$ with $\Psi(q,q')\leq\Psi(p,p')$, commutes.

 We inspect an relevant triple $(\tilde\alpha,t,s)$ in case (b), i.e.\ $(t,t'')=(p,p')$ is the largest index that occurs in $E(\tilde\alpha,t,s)$ and there is an arrow $\alpha':q'\to p'$ in $F^{-1}(\tilde\alpha)$ with $q'=s$. Then $(p,p')$ occurs in the term $w^{p,p'}M_{\alpha}\vert_{\beta_{p'}}w^{q',q'}$. Since $(p,p')$ is relevant, $p'\notin S_0$ and $\alpha'\notin S_1$. Thus we can assume that $M_{\alpha'}$ is the identity matrix. As in case (a), $M_{\alpha'}(\beta_{q'})\subset\beta_{p'}$ if $C_\beta^{F_\ast M}$ is not empty. Therefore, the identity matrix $(w_{b,b'})_{b,b'\in\beta_{p'}}$ occurs as a submatrix of $w^{q',q'}\vert_{\beta_{p'}}=M_{\alpha}\vert_{\beta_{p'}}w^{q',q'}$. This means that every equation that appears in the matrix of equations $E(\tilde\alpha,t,s)$ contains a linear term $w_{b,b'}$ where $b\in\cB_{p}$ and $b'\in\beta_{p'}$ and that all these linear terms are pairwise different. This allows to solve $E(\tilde\alpha,t,s)$ in $w^{p,p'}$ and the solution space is an affine space.

 For the same reasons as explained in case $(a)$, the claim of the theorem, restricted to the coordinates of the $w^{q,q'}$ with $\Psi(q,q')\leq\Psi(p,p')$, follows from the preceding.

 Finally, we have to observe that the exceptional cases \eqref{part1} and \eqref{part2} of Hypothesis (H) lead indeed to the two situations (a) and (b) as considered above. In the exceptional case \eqref{part1}, there is only one exceptional arrow $\tilde\alpha:\tilde p \to \tilde q$ that connects to $\tilde p$. Further there is an arrow $\alpha:p\to q$ in $F^{-1}(\tilde\alpha)$ such that $(\tilde\alpha,q,p')$ is of type 2a or 4a and such that $M_\alpha$ is the identity matrix. This means that the relevant pair $(p,p')$ appears only in $E(\tilde\alpha,q,p')$ as largest index. All assumptions that were made in the discussion of case (a) are satisfied for types 2a and 4a. Therefore, we can do the induction step for relevant pairs $(p'p')$ in the exceptional case \eqref{part1}.

In the exceptional case \eqref{part2}, there is only one exceptional arrow $\tilde\alpha:\tilde q \to \tilde p$ that connects to $\tilde p$. Further there is an arrow $\alpha':q'\to p'$ in $F^{-1}(\tilde\alpha)$ such that $(\tilde\alpha,p,q')$ is of type 2b or 3a. This means that the relevant pair $(p,p')$ appears only in $E(\tilde\alpha,p,q')$ as largest index. All assumptions that were made in the discussion of case (b) are satisfied for types 2b and 3a. Therefore, we can do the induction step for relevant pairs $(p'p')$ in the exceptional case \eqref{part2}.

This finishes the proof of the theorem.
\end{proof}

\subsection{Examples and non-examples}
\label{subsection: examples and counter examples}

In this section, we will consider some examples for Theorem \ref{thm: push-forward}. To start with, we will show two examples that fail to satisfy Hypothesis (H) as well as the claim of the theorem, which shows the significance of Hypothesis (H). Let $T$ be a quiver. A $T$-module $M$ is \emph{thin} if $\rk M_p\leq 1$ for every vertex $p$ of $T$, and $M$ is \emph{sincere} if $\rk M_p\geq 1$ for every vertex $p$ of $T$. 

\begin{ex}
 \label{ex: Kronecker quiver again}
 Example \ref{ex: kronecker quiver} provides an example of a push-forward module such that the according Schubert cells of the quiver Grassmannian are not reduced. We consider the representation $N$ of the Kronecker quiver $Q$ of dimension vector $(2,2)$ whose linear maps are the identity matrix and the Jordan block $J(0)$ with eigenvalue $0$. This $Q$-module is the push-forward $F_\ast M$ a representation $M$ of the tree $T$ along the  morphism
 \[
  \begin{tikzpicture}[description/.style={fill=white,inner sep=2pt}]
  \matrix (m) [matrix of math nodes, row sep=1em, column sep=2.5em, text height=1.5ex, text depth=0.25ex]
   { 2 & & 1 \\
     4 & & 3 && T \\ 
               \\
     \bullet & & \bullet && Q\\};
   \path[->,font=\scriptsize]
   (m-1-1) edge node[auto] {$\alpha_1$} (m-1-3)
           edge node[below] {$\gamma$} (m-2-3)
   (m-2-1) edge node[below] {$\alpha_2$} (m-2-3)
   (m-4-1.20) edge node[auto] {$\tilde\alpha$} (m-4-3.160)
   (m-4-1.340) edge node[below] {$\tilde\gamma$} (m-4-3.200)
   (m-2-5) edge node[auto] {$F$} (m-4-5) ;
  \end{tikzpicture}
 \]
 that maps $\alpha_1$ and $\alpha_2$ to $\tilde\alpha$ and $\gamma$ to $\tilde\gamma$. Set $S=\{1\}$ and define $M$ as the thin sincere $T$-module whose basis $\cB$ is identified with the vertex set $\{1,2,3,4\}$ of $T$ and whose linear maps send basis elements to basis elements.  Then $N=F_\ast M$. We use the obvious ordering of $\cB$ and consider $\beta=\{3,4\}$. 

 We re-calculate the Schubert cell $C_\beta^{F_\ast M}$ from Example \ref{ex: kronecker quiver}. Note that the equation $E(\tilde\alpha,t,s)$ is non-trivial  only if both $\beta_s$ and $\cB_t-\beta_t$ are non-empty. In this example, the Schubert cell $C_\beta^{F_\ast M}$ is thus determined as the subscheme of matrices
 \[
  \begin{pmatrix} w_{1,2} & 0 \\ 1 & 0 \\ 0 & w_{3,4} \\ 0 & 1 \end{pmatrix} 
 \]
 that satisfy the two Equations
 \begin{align*}
  E(\tilde\alpha,1,4) && w_{2,4} &= w_{1,3} \cdot 1 \\
  E(\tilde\gamma,1,4) && 0 &= w_{1,3} \cdot w_{2,4}. \\
 \end{align*}
 This means that $C_\beta^{F_\ast M}=\Spec\bigl(k[w_{2,4}]/(w_{2,4}^2)\bigr)$, which is a non-reduced scheme and thus not an affine space. Note that the fibre of $\tilde\gamma$ is of type 5, and therefore $F$ fails to satisfy Hypothesis (H).
\end{ex}

\begin{ex}
 \label{ex: cone as schubert cell}
 While the quiver Grassmannian of Example \ref{ex: Kronecker quiver again} has a reduced Schubert decomposition into affine spaces, the following is an example of a strictly ordered morphism with a triple of type 5 that leads to a reduced non-empty Schubert cell that is not an affine space. Consider 
 
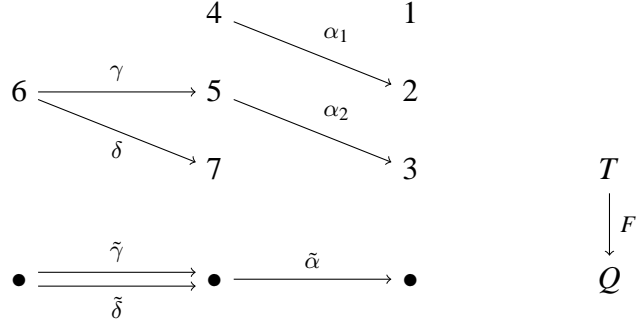
\[
  \begin{tikzpicture}[description/.style={fill=white,inner sep=2pt}]
  \matrix (m) [matrix of math nodes, row sep=1em, column sep=5em, text height=1.5ex, text depth=0.25ex]
   {   & 4 & 1 \\
     6 & 5 & 2 \\
       & 7 & 3 & T \\ 
                   \\
     \bullet & \bullet & \bullet & Q\\};
   \path[->,font=\scriptsize]
   (m-1-2) edge node[auto] {$\alpha_1$} (m-2-3)
   (m-2-2) edge node[auto] {$\alpha_2$} (m-3-3)
   (m-2-1) edge node[auto] {$\gamma$} (m-2-2)
           edge node[below] {$\delta$} (m-3-2)
   (m-5-1.20) edge node[auto] {$\tilde\gamma$} (m-5-2.160)
   (m-5-1.340) edge node[below] {$\tilde\delta$} (m-5-2.200)
   (m-5-2) edge node[auto] {$\tilde\alpha$} (m-5-3)
   (m-3-4) edge node[auto] {$F$} (m-5-4) ;
  \end{tikzpicture}
 \]
 and define $S=\{1,2,3\}$. Then $T$ is a tree extension of $S$. Define $M$ as the thin sincere $T$-module with ordered basis $\cB=\{1,\dotsc,7\}$ and identity matrices as morphisms. For the subset $\beta=\{2,3,7\}$, we have only one non-trivial equation
 \begin{align*}
  E(\tilde\alpha,1,7) && 0 &= w_{1,2} w_{4,7} + w_{1,3} w_{5,7},
 \end{align*}
 which shows that $C_\beta^{F_\ast M}$ is $\Spec\bigl(k[w_{1,2}, w_{1,3}, w_{4,7}, w_{5,7}]/(w_{1,2} w_{4,7} + w_{1,3} w_{5,7})\bigr)$, which is a reduced cone with a singularity at the origin. Thus $C_\beta^{F_\ast M, \red}$ is not isomorphic to an affine space.
\end{ex}

\begin{ex}
\label{ex: exceptional modules of the Kronecker quiver} 
 Consider the strictly ordered morphism
  
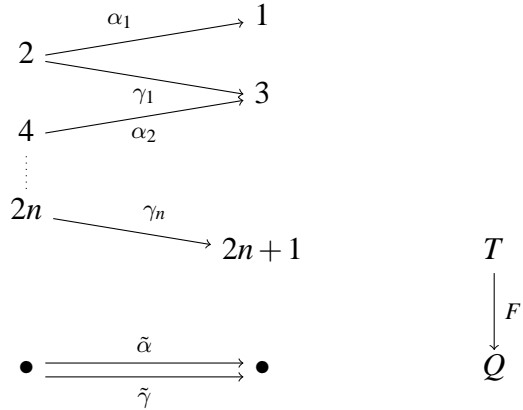
\[
  \begin{tikzpicture}[description/.style={fill=white,inner sep=2pt}]
  \matrix (m) [matrix of math nodes, row sep=0em, column sep=5em, text height=1.0ex, text depth=0.25ex]
   {  & 1 \\
     2    \\
      & 3 \\ 
     4    \\
         \ \\
     2n    \\
      & 2n+1 & T \\ 
              \ \\ \ \\
     \bullet & \bullet & Q\\};
   \path[->,font=\scriptsize]
   (m-2-1) edge node[auto] {$\alpha_1$} (m-1-2)
           edge node[below] {$\gamma_1$} (m-3-2)
   (m-4-1) edge node[below] {$\alpha_2$} (m-3-2)
   (m-6-1) edge node[auto] {$\gamma_n$} (m-7-2)
   (m-10-1.20) edge node[auto] {$\tilde\alpha$} (m-10-2.160)
   (m-10-1.340) edge node[below] {$\tilde\gamma$} (m-10-2.200)
   (m-7-3) edge node[auto] {$F$} (m-10-3) ;
   \path[dotted] (m-4-1) edge (m-6-1) ;
  \end{tikzpicture}
 \]
 and let $S$ be $\{1\}$ and $M$ the thin sincere $T$-module with basis $\cB=\{1,\dotsc,2n+1\}$ and whose linear maps are identity matrices. Then it is easily seen that $F:T\to Q$ satisfies Hypothesis (H), and therefore Theorem \ref{thm: push-forward} implies that $\Gr_\ue(F_\ast M)=\decomp C_\beta^{F_\ast M}$ is a decomposition into affine spaces for any dimension vector $\ue$.

 Note that $F_\ast M$ is a preprojective module of the Kronecker quiver $Q$ and all preprojective modules are of this form. Similarly, we find that the preinjective modules are push-forwards of a thin sincere $T$-module $M$ along a morphism
  
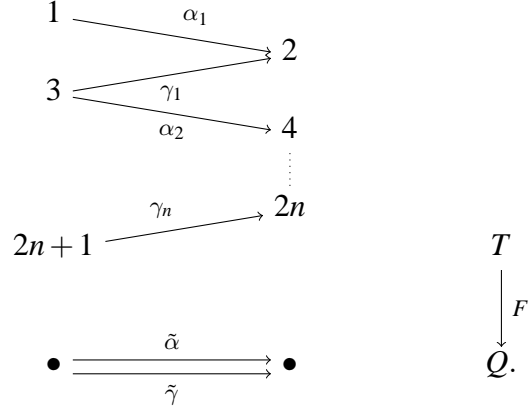
\[
  \begin{tikzpicture}[description/.style={fill=white,inner sep=2pt}]
  \matrix (m) [matrix of math nodes, row sep=0em, column sep=5em, text height=1.0ex, text depth=0.25ex]
   { 1    \\
      & 2 \\
     3    \\
      & 4 \\ 
         \ \\
      & 2n    \\
     2n+1 & & T \\ 
              \ \\ \ \\
     \bullet & \bullet & Q.\\};
   \path[->,font=\scriptsize]
   (m-1-1) edge node[auto] {$\alpha_1$} (m-2-2)
   (m-3-1) edge node[below] {$\gamma_1$} (m-2-2)
   (m-3-1) edge node[below] {$\alpha_2$} (m-4-2)
   (m-7-1) edge node[auto] {$\gamma_n$} (m-6-2)
   (m-10-1.20) edge node[auto] {$\tilde\alpha$} (m-10-2.160)
   (m-10-1.340) edge node[below] {$\tilde\gamma$} (m-10-2.200)
   (m-7-3) edge node[auto] {$F$} (m-10-3) ;
   \path[dotted] (m-4-2) edge (m-6-2) ;
  \end{tikzpicture}
 \]
 Therefore Theorem \ref{thm: push-forward} implies that $\Gr_\ue(F_\ast M)=\decomp C_\beta^{F_\ast M}$ is a decomposition into affine spaces for any dimension vector $\ue$. Since all indecomposable exceptional representation of the Kronecker quiver $Q$ are either preprojective or preinjective, we see that all indecomposable exceptional $Q$-modules $N$ have an ordered basis such that all Schubert cells $C_\beta(N)$ are affine spaces or empty. This recovers results from \cite{Caldero-Zelevinsky06}.
\end{ex}

\begin{ex}
 Another example that displays a typical situation that satisfies Hypothesis (H) is the following. Let $F:T\to Q$ be the morphism
  
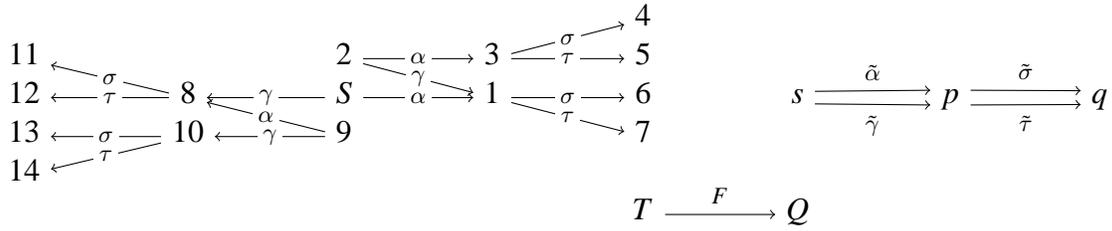
\[
  \begin{tikzpicture}[description/.style={fill=white,inner sep=2pt}]
  \matrix (m) [matrix of math nodes, row sep=0em, column sep=3.5em, text height=1.0ex, text depth=0.25ex]
   {    &    &   &   & 4  \\
     11 &    & 2 & 3 & 5  \\
     12 & 8  & S & 1 & 6 & s & p & q \\
     13 & 10 & 9 &   & 7 \\
     14 \\
       &&&& T & Q\\};
   \path[->,font=\scriptsize]
   (m-3-2) edge node[description] {$\sigma$} (m-2-1)
   (m-3-2) edge node[description] {$\tau$}   (m-3-1)
   (m-4-2) edge node[description] {$\sigma$} (m-4-1)
   (m-4-2) edge node[description] {$\tau$}   (m-5-1)
   (m-3-3) edge node[description] {$\gamma$} (m-3-2)
   (m-4-3) edge node[description] {$\alpha$} (m-3-2)
   (m-4-3) edge node[description] {$\gamma$} (m-4-2)
   (m-3-3) edge node[description] {$\alpha$} (m-3-4)
   (m-2-3) edge node[description] {$\alpha$} (m-2-4)
   (m-2-3) edge node[description] {$\gamma$} (m-3-4)
   (m-2-4) edge node[description] {$\sigma$} (m-1-5)
   (m-2-4) edge node[description] {$\tau$}   (m-2-5)
   (m-3-4) edge node[description] {$\sigma$} (m-3-5)
   (m-3-4) edge node[description] {$\tau$}   (m-4-5)
   (m-3-6.20) edge node[auto] {$\tilde\alpha$} (m-3-7.160)
   (m-3-6.340) edge node[below] {$\tilde\gamma$} (m-3-7.200)
   (m-3-7.20) edge node[auto] {$\tilde\sigma$} (m-3-8.160)
   (m-3-7.340) edge node[below] {$\tilde\tau$} (m-3-8.200)
   (m-6-5) edge node[auto] {$F$} (m-6-6);
  \end{tikzpicture}
 \]
 where the subquiver $S$ is mapped to the vertex $s$ of $Q$ and the arrows of $T-S$ are labelled by their image under $F$ (with the tilde removed). Consequently, the map $(T_0-S_0)\to Q_0$ sends $2$ and $9$ to $s$, it sends $1$, $3$, $8$ and $10$ to $p$ and it sends all other vertices to $q$.

 Note that $T$ is a tree extension of $S$. Let $M$ be a $T$-module with ordered basis $\cB$ such that for all $\alpha'\in T_1-S_1$, the matrix $M_{\alpha'}$ is the identity matrix. We assume that the ordering of $\cB$ coincides with the ordering of $T$, with all vertices of $S$ being smaller than $1$. Then it is easily verified that Hypothesis (H) holds in this situation, and we can apply Theorem \ref{thm: push-forward}. Therefore the Schubert cells $C_\beta^{F_\ast M}$ are of the form $C_{\beta_S}^{F_\ast M_S}\times \A^n$ for some $n\geq0$. Since $F_\ast M_S$ is supported by the one point quiver $\{s\}$, the Schubert cell $C_{\beta_S}^{F_\ast M_S}$ is nothing else than the Schubert cell of a usual Grassmannian, and therefore an affine space. 

 Altogether, this shows that for any dimension vector $\ue$ of $Q$, the Schubert decomposition $\Gr_\ue(F_\ast M)=\decomp C_\beta^{F_\ast M}$ is a decomposition into affine spaces.
\end{ex}

\begin{rem}
 Examples \ref{ex: Kronecker quiver again} and \ref{ex: cone as schubert cell} make clear why we have to consider the technical Hypothesis (H) in the proof of Theorem \ref{thm: push-forward}. Though a re-ordering of the vertices yield Schubert decompositions into affine spaces, it is not hard to to construct examples such that triples of type 5 cannot be avoided, and the according Schubert cells are not all affine spaces or empty. One can also construct examples with other fibre constellations that are not allowed in Hypothesis (H) such that the Schubert decomposition contains non-empty cells that are not affine spaces.
 
 From a combinatorial point of view, it is necessary to exclude certain constellations of the fibres of $F:T\to Q$. However Example \ref{ex: exceptional modules of the Kronecker quiver} gives hope that there is a good representation theoretic description of quiver Grassmannians with a Schubert decomposition into affine spaces. In particular one might raise the following question: does any exceptional indecomposable $Q$-module $M$ admit an ordered basis $\cB$ such that for every subset $\beta$ of $\cB$, the Schubert cell $C_\beta^M$ is an affine space or empty?
\end{rem}

\section{Consequences of the push-forward theorem}
\label{section: consequences of the push-forward theorem}

In this section, we will describe a series of consequences of Theorem \ref{thm: push-forward}. Whenever we have a Schubert decomposition of some quiver Grassmannian $\Gr_\ue(M)$ into affine spaces where $M$ is an $S$-module for some quiver $S$, we can use the Theorem \ref{thm: push-forward} to extend this result to a larger class of quiver Grassmannians. We formulate this method in the following statement.

\begin{cor}\label{cor: schubert decomposition of quiver grassmannians}
 Let $T$ be a tree extension of $S$ and $M$ a $T$-module with ordered basis $\cB$ that is ordered above $S$. Let $M_S$ be the restriction of $M$ to $S$ and $\cB_S=\cB\cap M_S$. Let $F:T\to Q$ be a morphism that satisfies Hypothesis (H). Denote by $M_S$ the restriction of $M$ to $S$. If $C_{\beta_S}^{F_\ast M}$ is an affine space or empty for every subset $\beta_S\subset \cB\cap M_S$, then
 \[
  \Gr_\ue (M) \quad = \quad \decomp_{\substack{\beta\subset\cB \text{ of type }\ue}} \ C_\beta^{F_\ast M}
 \]
 is a decomposition into affine spaces for any dimension vector $\ue$.
\end{cor}

\begin{proof}
 This follows immediately from Theorem \ref{thm: push-forward}.
\end{proof}

\begin{ex}
 Corollary \ref{cor: schubert decomposition of quiver grassmannians} allows us to expand Example \ref{ex: unions of one vertex quivers}. Let $T$ be a tree extension of $S$ and $M$ a $T$-module such that the restriction of $M$ to each connected component of $S$ is one of the quiver representations as considered in Examples \ref{ex: quiver with one vertex}--\ref{ex: union of two projective lines}. Let $\cB$ be an ordered basis of $M$ that induces an ordering of $T$ and that is ordered above $S$. Let $F:T\to Q$ a morphism that satisfies Hypothesis (H) and such that $F\vert_S:S\to Q$ is injective on vertices and arrows. Then 
 \[
  \Gr_\ue (M) \quad = \quad \decomp_{\substack{\beta\subset\cB \text{ of type }\ue}} \ C_\beta^{F_\ast M}
 \]
 is a decomposition into affine spaces for any dimension vector $\ue$.
\end{ex}

\subsection{Direct sums of modules}
\label{subsection: direct sums}

\begin{thm}\label{thm: direct sum}
 Let $T$ be a tree extension of $S$ and let $M^{(i)}$ be $T$-modules for $i=1,2$. Assume that $M_\alpha^{(i)}$ is an isomorphism for all arrows $\alpha$ of $T-S$ and for $i=1,2$. Define $M=M^{(1)}\oplus M^{(2)}$, and let $M_S$ and $M^{(i)}_S$ be the respective restrictions to $S$. Let $\cB_S$ be an ordered basis of $M_S$ such that $\cB_S=\cB_S^{(1)}\cup \cB_S^{(2)}$ where $\cB^{(i)}_S=\cB_S\cap\cM_i$ for $i=1,2$. Assume that $\cB_S$ induces an ordering of $S$. Let $\beta_S\subset\cB_S$ be a subset such that there is an $n_S\geq0$ with 
 \[
  C^{M_S}_{\beta_S} \quad \simeq \quad C_{\beta_S^{(1)}}^{M_S^{(1)}}\times C_{\beta_S^{(2)}}^{M_S^{(2)}}\times \A^{n_S}.
 \]  
 Then there is an extension $\cB$ of $\cB_S$ to $M$ that is ordered above $S$ such that for every subset $\beta\subset\cB$ with $\beta_S=\beta\cap\cB_S$, there is some $n\geq n_S$ such that 
 \[
  C^{M}_{\beta} \quad \simeq \quad C_{\beta^{(1)}}^{M^{(1)}}\times C_{\beta^{(2)}}^{M^{(2)}}\times \A^{n}
 \]
 if $C^M_\beta$ is not empty.
\end{thm}

\begin{proof}
 Define $T'=T\amalg T$ and $S'=S\amalg S\subset T'$. Then $T'$ is a tree extension of $S'$. Let $\iota_1:T\to T'$ and $\iota_2:T\to T'$ be the inclusions into the first resp.\ the second summand of $T'=T\amalg T$. Define the $T'$-module $M'=M^{(1)}\amalg M^{(2)}$ whose restriction to $\iota_1(T)$ is $M^{(1)}$ and whose restriction to $\iota_2(T)$ is $M^{(2)}$. By Lemma \ref{lemma: basis extensions for tree extensions}, we can extend $\cB_S$ to an ordered basis $\cB$ of $M'$ that is ordered above $S'$ and that satisfies $M_\alpha(\cB_p)\subset\cB_q$ for every arrow $\alpha:p\to q$ in $T'-S'$.

 Define $F:T'\to T$ by $F\vert_{\iota_1(T)}=\id_T=F\vert_{\iota_2(T)}$. It is easily verified that $F$ satisfies Hypothesis (H) (indeed only triples of types 1 and 2 occur for $Q$). By the very definition of push-forwards, $F_\ast M'=M$ and $F_\ast M_{S'}=M_S$. Further, we have $C_\beta^{M'}=C_{\beta^{(1)}}^{M^{(1)}}\times C_{\beta^{(2)}}^{M^{(2)}}$ and $C_{\beta_{S'}}^{M'_{S'}}=C_{\beta^{(1)}_S}^{M^{(1)}_S}\times C_{\beta_S^{(2)}}^{M_S^{(2)}}$ (by Lemma \ref{lemma: disjoint union}). This allows us to apply Theorems \ref{thm: tree extensions} and \ref{thm: push-forward} and the hypothesis of this theorem (to which we refer to by (h)) to conclude
 \begin{align*}
  C_\beta^M \ = \ C_\beta^{F_\ast M'} \ \underset{\ref{thm: push-forward}}= & \ C_{\beta_S}^{F_\ast M_S'} \ \times \ \A^{n_\beta} \ \times \ \A^{n_{F,\beta}} 
                                                                        & = & \ C_{\beta_S}^{M_S}         \ \times \ \A^{n_\beta} \ \times \ \A^{n_{F,\beta}} \\
                                                     \underset{\text{(h)}}= & \ C_{\beta^{(1)}_S}^{M^{(1)}_S} \ \times \ C_{\beta_S^{(2)}}^{M_S^{(2)}} \ \times \ \A^{n_S} \ \times \ \A^{n_\beta}\ \times\ \A^{n_{F,\beta}} \hspace{-8pt}      & = & \ C_{\beta_{S'}}^{M'_{S'}} \ \times \ \A^{n_S} \ \times \ \A^{n_\beta} \ \times \ \A^{n_{F,\beta}} \\
                                     \underset{\ref{thm: tree extensions}}= & \ C_{\beta}^{M'} \ \times \ \A^{n_S} \ \times  \ \A^{n_{F,\beta}} 
                                                                        & = & \ C_{\beta^{(1)}}^{M^{(1)}}\times C_{\beta^{(1)}}^{M^{(1)}} \ \times \ \A^{n_S}  \ \times \ \A^{n_{F,\beta}};
 \end{align*}
 thus the claim of the theorem is satisfied for $n=n_S+n_{F,\beta}$.
\end{proof}

\subsection{Representations of forests}
\label{subsection: monomial representations of forests}

Let $Q$ be a quiver and $M$ a $Q$-module over $k$. The \emph{support of $M$} is the subquiver $Q_M$ of $Q$ with vertices $Q_{M,0}=\{p\in Q|M_p\neq 0\}$ and edges $Q_{M,1}=\{\alpha\in Q_1|M_\alpha\neq0\}$. A \emph{forest} is a quiver $Q$ that is a union of trees. 

\begin{thm}\label{thm: monomial representations of forests}
 Let $Q$ be a forest, $M$ a $Q$-module and $\ue$ a dimension vector for $Q$. Assume that there is an ordered basis $\cB$ such that for all arrows $\alpha$ of $Q$, the matrix $M_\alpha$ is of the block form $\tinymat 0100$ where $1$ is a square identity matrix and the other three blocks are (possibly non-square) zero matrices. Then
 \[
  \Gr_\ue(M)  \quad = \quad \decomp_{\substack{\beta\subset\cB\text{ of type }\ue}} \ C_\beta^M
 \]
 is a decomposition into affine spaces.
\end{thm}

\begin{proof}
 Since all matrices are of the form $\tinymat 0100$ and $Q$ is a forest, it follows that $M$ decomposes into a direct sum $M\simeq\bigoplus_{i=1}^r M_i$ of thin and indecomposable $Q$-modules $M_i$. Let $Q_i$ be the support of $M_i$ for $i=1,\dotsc, r$ and $\iota_i:Q_i\to Q$ the inclusion. Then $Q_i$ is a tree and the restriction $M_{i,Q_i}$ of $M_i$ to $Q_i$ is a thin sincere $Q_i$-module with $M_\alpha(b_p)=b_q$ for every arrow $\alpha: p\to q$ of $Q_i$ and $\cB_p=\{b_p\}$ and $\cB_q=\{b_q\}$.

 We can further assume that every $Q_i$ has a vertex $p$ that connects to only one arrow $\alpha: p\to q$ and such that $e_{\iota_i(p)}=0$. If this is not the case, we add an arrow $\alpha: p\to q$ to $Q_i$ where $q$ is an arbitrary vertex of $Q_i$, which defines a tree $Q_i'$. We extend $Q$ to a forest $Q'$ that contains an arrow $\alpha': p'\to \iota(q)$, which allows us to extend $\iota_i$ to an inclusion $\iota':Q_i'\to Q'$ that maps $\alpha$ to $\alpha'$. We extend $M_{i,Q_i}$ to the $Q_i'$-module $M'_{i,Q'_i}$with basis $\cB'=\cB\cup\{b_p\}$ by $M'_{i,p}=k$ and $M'_{i,\alpha}(b_p)=b_q$ where $\cB_q=\{b_q\}$. Extend $M_i$ to the $Q'$-module $M'_i$ whose restriction to $Q_i'$ is $M'_{i,Q'_i}$. Extend all other direct summands $M_j$ of $M$ to the $Q'$-module $M'_j$ with $M'_{j,\alpha'}: 0\to M'_{j,\iota_i(q)}$ and define $M'=\bigoplus_{j=1}^r M'_j$. Define the dimension vector $\ue'$ for $Q'$ by $e'_p=0$ and $e'_{p'}=e_{p'}$ for $p'\in Q_0$. Then $Q'$, $M'$ and $\ue'$ satisfy the hypothesis of the theorem 
and $\Gr_{\ue'}(M')$ is the same as $\Gr_\ue(M)$. Therefore we can assume the existence of the vertex $p$ in $Q_i$.

 Define $T=\coprod_{i=1}^r Q_i$, which is a forest, and $S=\{p_1,\dotsc,p_r\}$ where $p_i$ is a vertex of $Q_i$ with the properties from the last paragraph. Define $N=\coprod_{i=1}^r M_{i,Q_i}$. Let $F:T\to Q$ be the morphism of quivers that restricts to $\iota_i: Q_i\to Q$ for each connected component $Q_i$ of $T$. Then $F$ is strictly ordered. Since the matrices $M_\alpha$ are of the form $\tinymat 0100$, all relevant triples $(\alpha,t,s)$ are of type 0, 1 or 2.

 Since the Schubert decomposition of $\Gr_\ue(M)$ only depends on the ordering of all the sets $\cB_p$ for $p\in Q$, but not on the ordering of $b\in\cB_p$ and $b'\in\cB_q$ if $p\neq q$, we can reorder $\cB$ such that we preserve the orderings of all subsets $\cB_p$ for $p\in Q$, but such that $\cB$ is ordered above $S$ (this is basically explained in Lemma \ref{lemma: basis extensions for tree extensions}). With respect to this new ordering, the morphism $F:T\to Q$ satisfies Hypothesis (H).

 Therefore we can apply Theorem \ref{thm: push-forward} to obtain for every subset $\beta$ of $\cB$ and $\beta_S=\beta\cap M_S$ that $C_\beta^M\simeq C_{\beta_S}^{M_S}\times \A^{n_\beta}$ for some $n_\beta\geq 0$, provided $C_\beta^M$ is not empty. Since $e_p=0$ for all $p\in S$, the set $\beta_S$ is empty if $\beta$ is of type $\ue$, which means that $C_{\beta_S}^{M_S}=\Spec k$ is a point. Therefore  $C_\beta^M$ is an affine space for every $\beta\subset\cB$ of type $\ue$. This completes the proof of the theorem.
\end{proof}

\begin{ex}[Degenerate flag varieties]\label{ex: degenerate_flag_varieties1}
 As a particular class of representation of trees whose matrices are of the form $\tinymat 0100$, we re-obtain the result \cite[Thm.\ 7.11]{Cerulli-Feigin-Reineke11}, which says that degenerate flag varieties have a decomposition into affine spaces. Indeed, the results of \cite{Cerulli-Feigin-Reineke11} are stronger since the decomposition is given by a group action. We inspect degenerate flag varieties in more detail in Example \ref{ex: degenerate_flag_varieties2}
\end{ex}

\section{The cohomology of quiver Grassmannians}
\label{section: cohomology of quiver grassmannians}

A Schubert decomposition of a quiver Grassmannian into affine spaces yields certain information about the cohomology of the quiver Grassmannian. We concentrate on the singular cohomology of a complex quiver Grassmannian, i.e.\ the case $k=\C$. Similar arguments can be used to treat the $l$-adic cohomology with proper support of quiver Grassmannians over the integers.
 
The basic fact that we will use is the following, cf.\ Lemma 6 in Appendix B of \cite{Fulton97}. Let $X$ be a smooth projective $k$-scheme of complex dimension $m$ that has a decomposition $X=\decomp_{i\in I}Z_i$ into affine spaces $Z_i$ such that there is a series of closed subschemes $X^{(0)}\subset X^{(1)}\subset \dotsb \subset X^{(n)}=X$ of $X$ such that all $l\in\{0,\dotsc,n\}$, there is a subset $I_l\subset I$ such that $X^{(l)}-X^{(l-1)}=\coprod_{l\in I_l} Z_i$ (as a disjoint union of schemes). Then the cohomology classes of the closures $\overline{Z_i}$ form a $\Z$-basis of the cohomology ring $H^\ast(X,\Z)$. If $Z_i\simeq\A^{d_i}$, then the class $[\overline{Z_i}]$ is an element of $H^{2m-2d_i}(X,\Z)$. In particular, the odd cohomology of $X$ vanishes. 

\begin{lemma}
 Let $Q$ be a quiver and $M$ a $Q$-module with ordered basis $\cB$. Let $\ue$ be a dimension vector for $Q$. Then there are subsets $I_0\subset\dotsb\subset I_n=I$ of $I=\{\beta\subset\cB|\beta\text{ is of type }\ue\}$ such that $X^{(l)}=\bigcup_{\beta\in I_l} C_\beta^M$ is a closed subscheme of $\Gr_\ue(M)$ for all $l=0,\dotsc,n$ and such that $X^{(l)}-X^{(l-1)}$ is isomorphic to the disjoint union $\coprod_{\beta\in I_l-I_{l-1}} C_\beta^M$ for all $l=1,\dotsc,n$.
\end{lemma}

\begin{proof}
 Define $I_0$ as the set of all subsets $\beta$ of $\cB$ of type $\ue$ such that $C_\beta^M$ is a closed subscheme of $\Gr_\ue(M)$. Since there are only finitely many subsets $\beta$ of $\cB$, $X^{(0)}=\coprod_{\beta\in I_0}C_\beta^M$ is a closed subscheme of $\Gr_\ue(M)$.

 If $I_{l-1}$ is defined for $l>0$, then we define $I_l$ as the set of all subsets $\beta$ of $\cB$ of type $\ue$ such that the complement $\overline{C_\beta^M}-C_\beta^M$ of $C_\beta^M$ in its own closure is contained in $X^{(l-1)}$. Then $X^{(l)}=\bigcup_{\beta\in I_l} C_\beta^M$ is a closed subscheme of $\Gr_\ue(M)$ and $X^{(l)}-X^{(l-1)}$ is isomorphic to the disjoint union $\coprod_{\beta\in I_l-I_{l-1}} C_\beta^M$.

 The proof is finished once we have shown that there is an $n$ such that $I_n=I$ and therefore $X^{(n)}=\Gr_\ue(M)$. Since the Schubert cells of $C_\beta^M$ are defined as the pull-back of the Schubert cells of the product Grassmannian $\prod_{i\in Q_0} \Gr(e_i,\dim M_i)$, the intersection $C_\beta^M\cap \overline{C_{\beta'}^M}$ of a Schubert cell with the closure of another Schubert cell is non-trivial only if $\beta\preceq\beta'$  (cf.\ Section \ref{section: schubert cells} for the definition of $\beta\preceq\beta'$).
 
 This implies for $l>0$ that if $\beta$ is a minimal element of the partial ordered set $I-I_{(l)}$ that $\overline{C^M_\beta}\subset C_\beta^M\cup X^{(l-1)}$. This means that $\beta\in I_l$ by the definition of $I_l$. Therefore the sequence $I_0\subset \dotsc\subset I_{l-1}\subset I_l \subset\dotsc $ is properly growing as long as there are (minimal) elements in $I-I_l$. Since $I$ is finite, there is an $n$ such that $I_n=I$.
\end{proof}

\begin{cor} \label{cor: cohomology for smooth quiver grassmannians}
  Let $Q$ be a quiver and $M$ a $Q$-module with ordered basis $\cB$. Let $\ue$ be a dimension vector for $Q$. If $\Gr_\ue(M)$ is smooth and $\Gr_\ue(M)=\decomp C_\beta^M$ is a decomposition into affine spaces, then the cohomology classes of the closures $\overline{C_\beta^M}$ form a $\Z$-basis of the cohomology ring $H^\ast(\Gr_\ue(M),\Z)$. If $d_\beta=\dim C_\beta^M$, then the class $[\overline{C_\beta^M}]$ is an element of $H^{2m-2d_\beta}(X,\Z)$. In particular, the odd cohomology of $X$ vanishes. \qed
\end{cor}

The Euler characteristic of a complex scheme is additive in decompositions into locally closed subschemes. This means that the Schubert decomposition $\Gr_\ue(M)=\decomp C_\beta^M$ yields the formula
\[
 \chi\bigl( \Gr_\ue(m)\bigr) \quad = \quad \sum_{\beta\subset\cB\text { of type }\ue} \ \chi\bigl( C_\beta^M \bigr).
\]
Note further that $\chi(C_\beta^M)=\chi(C_\beta^{M,\red})$ and that the Euler characteristic of an affine space is $1$. Therefore we have the following result without any assumptions on the smoothness of $\Gr_\ue(M)$.

\begin{prop}
 \label{prop: euler characteristic from dias}
 Let $Q$ be a quiver, $M$ a $Q$-module with ordered basis $\cB$ and $\ue$ a dimension vector for $Q$. Assume that the reduced Schubert decomposition $\Gr_\ue(M)=\decomp C_\beta^{M,\red}$ is a decomposition into affine spaces. Then the Euler characteristic of $\Gr_\ue(M)$ equals the number of non-empty Schubert cells $C_\beta^{M}$ where $\beta$ is of type $\ue$.
\end{prop}

\begin{rem}
 \label{rem: on the results of cerulli and haupt}
 Together with the characterization of non-empty Schubert cells in Theorems \ref{thm: tree extensions} and \ref{thm: push-forward}, we recover Theorem 1 of \cite{Cerulli11} and Corollary 3.1 of \cite{Haupt12} under the assumption of Hypothesis (H).
\end{rem}

\subsection{Regular decompositions}
\label{subsection: regular decompositions}

The following definition makes sense for all rings $k$. 
Let
\[
 \varphi: \ \coprod_{i\in I} \ C_i \quad \longrightarrow \quad X
\]
be a decomposition of $X$ into locally closed subschemes $C_i$. A locally closed subscheme $Z$ of $X$ \emph{decomposes (w.r.t.\ $\varphi$)} if there is a subset $I_Z\subset I$ such that $\varphi$ restricts to a decomposition 
\[
 \coprod_{i\in I_X} \ C_i \quad \longrightarrow \quad Z.
\]
Note that this is a purely topological property of $Z$. Note further that $I_Z$ is uniquely determined if all cells $Z_i$ are non-empty. The decomposition $\varphi:\decomp Z_i\to X$ is \emph{regular} if the closures of all cells $Z_i$ (for $i\in I$) decompose w.r.t.\ $\varphi$. This extends the notion of a regular torification from Section 6.2 of \cite{LL09}. 

The relevance of regular decomposition for Schubert calculus is that they admit a way to calculate the product of cohomology classes in the cohomology ring of the irreducible components of the quiver Grassmannian.

\begin{lemma}\label{lemma: cohomology for regular decompositions into affine spaces}
 Let $X$ be a complex projective scheme of dimension $m$ with irreducible components $X_1,\dotsc,X_n$. Let $\varphi:\coprod_{i\in I} \A^{d_i}\to X$ be a regular decomposition into affine spaces. Then the embeddings $\iota_i: X_i\to X$ define a graded inclusion 
        \[
         (\iota_1^\ast,\dotsc,\iota_n^\ast): \quad H^\ast(X,\Z) \quad \stackrel{}\longrightarrow \quad \bigoplus_{l=1}^n \ H^\ast(X_l,\Z)
        \]
        of graded rings.
\end{lemma}

\begin{proof}

 We make an induction on the number $n$ of irreducible components of $X$. If $n=1$, then \eqref{part3} is trivial.

 Let $n>1$. Then $X$ is the union of two closed subsets $Y$ and $X_n$ for $Y=X_1\cup\dotsc\cup X_{n-1}$. Let $\imath_Y:Y\to X$ and $\imath_X:X_n\to X$ be the embeddings of $Y$ resp.\ $X_n$ into $X$ and let $\jmath_Y:Z\to Y$ and $\jmath_X:Z\to X_n$ be the embeddings of $Z=Y\cap X_n$ into $Y$ resp.\ $X_n$. We consider the Mayer-Vietoris sequence
 \[
  \begin{split}
    \dotsc \quad & \longrightarrow \quad H^d(X,\Z) \quad \stackrel{(\imath_Y^\ast,\imath_{X}^\ast)}\longrightarrow \quad H^d(Y,\Z)\oplus H^d(X_n,\Z) \quad \stackrel{(\jmath_Y^\ast,-\jmath_{X}^\ast)}\longrightarrow \quad H^d(Z,\Z) \\
    \quad &\longrightarrow \quad H^{d+1}(X,\Z) \quad \longrightarrow \quad \dotsc
  \end{split}
 \]
 Let $X_l$ be an irreducible component of $X$ and $C_i=\A^{d_i}$ a cell of the decomposition $\varphi$. If $X_l$ intersects $C_i$ non-trivially, then $C_i$ is contained in $X_l$ since $C_i$ is irreducible. This shows that $X_l$ decomposes w.r.t.\ $\varphi$. Since $X_l$ is closed in $X$, it contains also the closure $\overline{C_i}$ of each of its cell $C_i\subset X_l$. It follows that for all $\{l_1,\dotsc,l_r\}\subset\{1,\dotsc,n\}$, the intersection $X_{l_1}\cap\dotsb\cap X_{l_r}$ decomposes w.r.t.\ $\varphi$ and contains the closure of each of its cells $C_i$.

 If $[\overline{C_i}]$ is the cohomology class of the closure of $C_i$, then the homomorphism $\iota_Y^\ast:H(X,\Z)\to H(Y,\Z)$ sends $[\overline{C_i}]$ to the class of $\overline{C_i}$ if $\overline{C_i}$ is a subscheme of $Y$, or to $0$ if not. The analog statement is true for $\imath_X$, $\jmath_Y$ and $\jmath_X$. By the induction hypothesis, $H^{d}(Y,\Z)$, $H^{d}(X_n,\Z)$ and $H^{d}(Z,\Z)$ are freely generated by the classes of the closures of the cells $C_i$ of dimension $d'$ that are contained in $Y$, $X_n$ resp.\ $Z$ if $d=2m-2d'$ is even, and they are $0$ if $d$ is odd. Therefore the homomorphism $(\jmath_Y^\ast,-\jmath_{X}^\ast):H^d(Y,\Z)\oplus H^d(X_n,\Z)\to H^d(Z,\Z)$ is surjective for every degree $d$, which means that the Mayer-Vietoris sequence splits into short exact sequences
 \[
    0 \quad \longrightarrow \quad H^d(X,\Z) \quad \stackrel{(\imath_Y^\ast,\imath_{X}^\ast)}\longrightarrow \quad H^d(Y,\Z)\oplus H^d(X_n,\Z) \quad \stackrel{(\jmath_Y^\ast,-\jmath_{X}^\ast)}\longrightarrow \quad H^d(Z,\Z) \quad \longrightarrow \quad 0.
 \]
 Since both $\imath^\ast_Y:H^\ast(X,\Z)\to H^\ast(Y,\Z)$ and $\imath_X^\ast:H^\ast(X,\Z)\to H^\ast(X_n,\Z)$ are ring homomorphisms by the induction hypothesis, \eqref{part3} follows. This completes the proof of the lemma.
\end{proof}

\begin{rem}
 The inclusion  $H^\ast(X,\Z) \to \bigoplus_{l=1}^n \ H^\ast(X_l,\Z)$ is the initial part of an exact sequence of $\Z$-modules of the form
 \begin{multline*}
  0 \quad \longrightarrow \quad H^\ast(X,\Z) \quad \stackrel{(\iota_1^\ast,\dotsc,\iota_n^\ast)}\longrightarrow \quad \bigoplus_{l=1}^n \ H^\ast(X_l,\Z) \quad \longrightarrow \quad \bigoplus_{1\leq l_1 < l_2 \leq n} \ H^\ast(X_{l_1,l_2},\Z)  \quad \longrightarrow \quad \\
  \dotsc  \quad \longrightarrow \quad \bigoplus_{1\leq l_1 < \dotsc < l_{n-1} \leq n} \ H^\ast(X_{l_1,\dotsc,l_{n-1}},\Z) \quad \longrightarrow \quad H^\ast(X_{1,\dotsc,n},\Z) \quad \longrightarrow \quad 0
 \end{multline*}
 where $X_{l_1,\dotsc,l_r}=X_{l_1}\cap \dotsc \cap X_{l_r}$ and the homomorphisms are defined as alternating sum of restriction maps $H^\ast(X_{l_1,\dotsc,l_r},\Z)\to H^\ast(X_{l_1,\dotsc,l_{r+1}},\Z)$, similar to those that appear in the definition of singular cohomology or \v{C}ech cohomology.
\end{rem}

\subsection{Examples and conjectures}
\label{subsection: examples and conjectures}

In this section, we will describe some examples (and counter examples) of Schubert decompositions of quiver Grassmannians that are regular. Everything can be considered over an arbitrary base ring $k$.

\begin{ex}[Usual Grassmannians and flag varieties]\label{ex: regular schubert decompositions of flag varieties}
 It is well-known that the Schubert decomposition of a usual Grassmannian or, more generally, of a flag variety is regular (cf.\ Exercise 13 of \S 9.4 and p.\ 159 in \cite{Fulton97}). In our notation, this fact takes the following shape.

 Let $\ue=(e_1,\dotsc,e_r)$ be the type of the flag variety $X$ of subspaces in $k^m$. Let $Q$ be a quiver of the form $1\to\dotsb\to r$ and $M$ the $Q$-module $k^m\stackrel\id\longrightarrow \dotsb \stackrel\id\longrightarrow k^m$. Then $\Gr_\ue(M)$ is isomorphic to $X$. If we order the standard basis $\cB=\{b_{k,p}\,|\,k=1\dotsc,m;\,p=1,\dotsc,r\}$ of $M$ lexicographically, then the decomposition 
 \[
  \Gr_\ue(M) \quad = \quad \decomp_{\beta\subset\cB\text{ of type }\ue} C_\beta^M
 \]
 coincides with the usual decomposition of $X$ into Schubert cells, cf.\ Example \ref{ex: flag varieties}.
 
 There is a natural action of $\GL_m$ on $\Gr_{e_p}(M_p)$ for each $p\in Q_0$, and thus a diagonal action on the flag variety $\Gr_\ue(M)$. The orbits of the upper triangular Borel subgroup $B$ of $\GL_m$ coincide with the Schubert cells $C_\beta^M$. Since the closure of an orbit is decomposes into orbits, the Schubert decomposition of $\Gr_\ue(M)$ is regular. More precisely, we have
 \[
  \overline{C_\beta^M} \quad = \quad \decomp_{\gamma\preceq\beta} \ C_\gamma^M.
 \]
\end{ex}

\begin{ex}[Representations of forests]\label{ex: regular schubert decompositions of representations of forests}
 Let $Q$ be a forest with $\kappa$ vertices and $M$ a $Q$-module with ordered basis $\cB$ such that $M_\alpha$ is the identity matrix for all arrows $\alpha$ of $Q$. By Theorem \ref{thm: Grassmannian fibrations}, there is a sequence $Q^{(1)}\subset\dotsb\subset Q^{(\kappa)}=Q$ of subquivers and a sequence
 \[
 \xymatrix{\Phi:\quad \Gr_\ue(M) \quad \ar@{->>}[r]^(0.6){\varphi_\kappa} & \quad \dotsb \quad \ar@{->>}[r]^(0.35){\varphi_{2}} & \quad \Gr_{\ue^{(1)}}(M^{(1)}) \quad \ar@{->>}[r]^(0.58){\varphi_{1}} & \quad \Spec k}
 \]
 of fibre bundles $\varphi_i$ whose fibres are Grassmannians $\Gr(\tilde e_i,\tilde m_i)$ for certain $\tilde e_i\leq\tilde m_i$ and $i=1,\dotsc,\kappa$. Here $M^{(i)}$ and $\ue^{(i)}$ are the restrictions of $M$ resp.\ $\ue$ to $Q^{(i)}$. Since every fibre has a regular Schubert decomposition such that the closure of each cell decomposes into the cells with smaller index, we obtain
 \[
  \overline{C_\beta^M} \quad = \quad \decomp_{\gamma\preceq\beta} \ C_\gamma^M,
 \]
 which generalizes Example \ref{ex: regular schubert decompositions of flag varieties}.
\end{ex}

\begin{ex}[Degenerate flag varieties]\label{ex: degenerate_flag_varieties2}
 Cerulli, Feigin and Reineke identify in \cite{Cerulli-Feigin-Reineke11} degenerate flag varieties of Dynkin type with certain quiver Grassmannians $\Gr_\ue(M)$ and establish a regular decomposition into the finitely orbits of the action of a certain Borel subgroup of the automorphism group of $M$. 

 In this example, we consider the case of a complete degenerate flag variety $F_\ue^a$ of flags of type $\ue=(1,\dotsc,n)$ in $k^{n+1}$, which can be identified with the quiver Grassmannian $\Gr_\ue(P\oplus I)$ where $P$ is the direct sum over all indecomposable projective $Q$-modules and $I$ is the direct sum over all indecomposable injective $Q$-modules for the equioriented quiver $Q$ of type $A_n$. In Section 7.2 of \cite{Cerulli-Feigin-Reineke11}, the reader finds a detailed description of the orbits of $B$ for this case. We will see that this decomposition coincides indeed with the Schubert decomposition w.r.t.\ a certain choice of ordered basis. It seems to be interesting to work out the connection for the general degenerate flag variety of Dynkin type.

 Let $Q=1\to\dotsb\to n$ be the underlying equioriented quiver of type $A_n$. For $i=1,\dotsb,n$, let $P_i$ be the indecomposable projective $Q$-module with support $i\to\dotsb\to n$ and let $I_i$ be the indecomposable injective $Q$-module with support $1\to \dotsb\to i$. Then $P=\bigoplus_{i=1}^n P_i$ and $I=\bigoplus_{i=1}^n I_i$.
 
 The $k$-module $P_{i,j}$ is trivial if $j<i$ and of rank one if $j\geq i$, in which case we denote the corresponding basis vector by $b^P_{i,j}$. The $k$-module $I_{i,j}$ is trivial if $j>i$ and of rank one if $j\leq i$, in which case, we denote the corresponding basis vector by $b^I_{i,j}$. The set 
 \[
  \cB \quad = \quad \{\ b^P_{i,j} \ | \ 1\leq i\leq j\leq n \ \} \quad \cup \quad \{ \ b^I_{i,j} \ | \ 1\leq j\leq i\leq n \ \}
 \]
 is a basis for $M=P\oplus I$. 

 We order $\cB$ as follows. First note that the relative order of the subsets $\cB_i$ (where $i\in Q_0$) is irrelevant for the shape of the Schubert cell $C_\beta^M$ of $\Gr_\ue(M)$, cf.\ Remark \ref{rem: irrelevance of oredering of basis elements of different vertices}. For a vertex $i$, we order $\cB_i$ by $b^I_{i,i}<\dotsb<b^I_{i,r}<b^P_{i,1}<\dotsb<b^P_{i,i}$. Then the Schubert cell $C_{\beta_i}^{M_i}$ coincides with the cell $C_{L_i}$ of \cite{Cerulli-Feigin-Reineke11} (as defined in Section 7.2), and $C_\beta^M$ coincides with the intersection of $\Gr_\ue(M)$ with the product $\prod_{i=1}^r C_{L_i}$ in $\Gr_\ue(\um)$. By Theorem 7.11 in \cite{Cerulli-Feigin-Reineke11}, this cell coincides with an orbit of the action of $B$ on $\Gr_\ue(M)$.

 In the notation of this paper, we can identify this Schubert decomposition with the Schubert decomposition of $\Gr_\ue(M')$ w.r.t.\ $\cB$ where $M'$ is the $Q$-module
 \[
  k^{n+1} \quad \stackrel {J(0)}\longrightarrow \quad k^{n+1} \quad \stackrel {J(0)}\longrightarrow \quad \dotsb \quad \stackrel {J(0)}\longrightarrow \quad k^{n+1}
 \]
 and $\cB$ is the standard ordered basis of $M'$ (recall that $J(0)$ is a maximal Jordan block with $0$ on the diagonal and $1$ on the upper side diagonal) .
\end{ex}

Based on the last two examples and further calculations, I expect that the following statements are true.

\begin{conj}\label{conj: regular decompositions for monomial representations of forests}
 Let $Q$ be a forest and $M$ a $Q$-module with ordered basis $\cB$ such that for every arrow $\alpha$ of $Q$, the matrix $M_\alpha$ is a block matrix $\tinymat 0I00$ where $I$ is a square identity matrix. Then $\Gr_\ue(M)=\coprod C_\beta^M$ is a regular decomposition into affine spaces for every dimension vector $\ue$.
\end{conj}

\begin{conj}\label{conj: regular decomposition for tree extensions}
 Let $T$ be a tree extension of $S$ and $M$ a $T$-module such that $M_\alpha$ is an isomorphism for all all arrows $\alpha$ in $T-S$. Let $M_S$ be the restriction of $M$ to $S$ and $\cB$ be an ordered basis of $M$ that is ordered above $\cB_S=\cB\cap M_S$. Let $\beta\subset\cB$ of type $\ue$ and $\beta_S=\beta\cap M_S$ of type $\ue_S$. Then the following holds true.
 \begin{enumerate}
  \item If $\overline{C_{\beta_S}^{M_S}}$ decomposes into Schubert cells, then also $\overline{C_\beta^M}$ decomposes into Schubert cells. 
  \item If $\overline{C_{\beta_S}^{M_S}}$ decomposes into Schubert cells and if
                     \[
                      \overline{C_{\beta_S}^{M_S}} \quad = \quad \decomp_{\gamma_S\preceq\beta_S\text{ of type }\ue_S} \ C_{\beta'_S}^{M_S},   \qquad \text{then}\qquad 
                      \overline{C_{\beta}^{M}} \quad = \quad \decomp_{\beta'\preceq\beta\text{ of type }\ue} \ C_{\beta'}^{M}.
                     \]    
 \end{enumerate}
 In particular, if $\Gr_{\ue_S}(M_S)=\decomp C_{\beta_S}^{M_S}$ is a regular decomposition, then so is $\Gr_{\ue}(M)=\decomp C_{\beta}^{M}$.
\end{conj}

\begin{small}
 \bibliographystyle{plain}

\end{small}

\end{document}